\pdfoutput=1
\documentclass[11pt]{amsart}
\usepackage[T1]{fontenc}
\usepackage{amsmath}
\usepackage{amssymb}
\usepackage{amsthm}
\usepackage{pdfpages}
\usepackage{textcomp}
\usepackage{graphicx}
\usepackage{tikz}
\usetikzlibrary{patterns}
\usepackage{enumerate}
\usepackage[justification=centering]{caption}

\newtheorem{theorem}{Theorem}[section]
\newtheorem{definitio}[theorem]{Definition}

\newtheorem{lemma}[theorem]{Lemma}
\newtheorem{proposition}[theorem]{Proposition}
\newtheorem{corollary}[theorem]{Corollary}

\newtheorem{question}[theorem]{Question}
\newtheorem*{lem}{Lemma}

\newtheorem*{theo}{Theorem}

\newcommand{\Z}{\mathbb{Z}}

\newcommand{\OO}{\mathcal{O}}

\newcommand{\lk}{\textrm{\textnormal{lk}}}
\newcommand{\Int}{\mathrm{Int}}
\newcommand{\Arf}{\mathrm{Arf}}
\newcommand{\sd}{s_{\scriptscriptstyle{\Delta}}}
\newcommand{\rd}{r_{\scriptscriptstyle{\Delta}}}

\newcommand{\draww}[1]{\draw[white,line width=5pt] #1 \draw #1}
\newcommand{\tour}{ 
 \draw[very thick] (0,2) arc (90:270:2) -- (1.5,-2);
 \draw[very thick,dashed] (1.5,-2) node[above] {$\partial\Sigma$} -- (2.5,-2) (0,2) -- (1,2);}

\definecolor{vert}{RGB}{0,205,0}

\title{Slice genus, $T$--genus and $4$--dimensional clasp number}
\author{Delphine Moussard}

\begin{document}

\begin{abstract}
 The $T$--genus of a knot is the minimal number of borromean-type triple points on a normal singular disk with no clasp bounded by the knot; it is an upper bound for the slice genus. Kawauchi, Shibuya and Suzuki characterized the slice knots by the vanishing of their $T$--genus. We generalize this to provide a $3$--dimensional characterization of the slice genus. Further, we prove that the $T$--genus majors the $4$--dimensional positive clasp number and we deduce that the difference between the $T$--genus and the slice genus can be arbitrarily large. We introduce the ribbon counterpart of the $T$--genus and prove that it is an upper bound for the ribbon genus. Interpreting the $T$--genera in terms of $\Delta$--distance, we show that the $T$--genus and the ribbon $T$--genus coincide for all knots if and only if all slice knots are ribbon. We work in the more general setting of algebraically split links and we also discuss the case of colored links. Finally, we express Milnor's triple linking number of an algebraically split $3$--component link as the algebraic intersection number of three immersed disks bounded by the three components.
\end{abstract}

\maketitle

\tableofcontents

\section{Introduction}

The first attempts to prove the ribbon--slice conjecture led to a 3--dimensional characterization of slice knots. Kawauchi, Shibuya and Suzuki \cite{KSS} proved that slice knots are exactly those that bound a normal singular disk in the 3--sphere with no clasp and no triple point of a certain type, called here borromean. We generalize this result to a 3--dimensional characterization of the slice genus, proving that the slice genus of a knot is the minimal genus of a normal singular surface in the 3--sphere, with no clasp and no borromean triple point, bounded by the knot.

It was proved by Kaplan \cite{Kaplan} that any knot bounds a normal singular disk with no clasp. This allowed Murakami and Sugishita \cite{MS} to define the $T$--genus of a knot as the minimal number of borromean triple points on such a disk. They proved that the $T$--genus is a concordance invariant and an upper bound for the slice genus. Further, they showed that the mod--$2$ reduction of the $T$--genus coincides with the Arf invariant. From these properties, they deduced the value of the $T$--genus for several knots for which the difference between the $T$--genus and the slice genus is 0 or 1. This raises the question of wether this difference can be greater than one; we prove that it can be arbitrarily large. For this, we show that the $T$--genus is an upper bound for the 4--dimensional positive clasp number and we use a recent result of Daemi and Scaduto \cite{DaSca} that states that the difference between the 4--dimensional positive clasp number and the slice genus can be arbitrarily large. 

We introduce the ribbon $T$--genus of a knot, defined as the minimal number of borromean triple points on an immersed disk bounded by the knot with no clasp and no non-borromean triple point. In \cite{KMS}, Kawauchi, Murakami and Sugishita proved that the $T$--genus of a knot equals its $\Delta$--distance to the set of slice knots. We prove the ribbon counterpart of it, namely that the ribbon $T$--genus of a knot equals its $\Delta$--distance to the set of ribbon knots. As a consequence, providing a knot with distinct $T$--genus and ribbon $T$--genus would imply the existence of a non-ribbon slice knot.

We generalize the definition and properties of the $T$--genus to algebraically split links. In addition, we express Milnor's triple linking number of an algebraically split 3--component link as the algebraic intersection of three disks bounded by the three components, that intersect only along ribbons and borromean triple points. We also give an elementary proof, in the setting of non-split links, that the difference between the $T$--genus and the slice genus can be arbitrarily large, computing the $T$--genus on a family of cabled borromean links. Finally, we discuss the case of colored links.

\vspace{1ex}

\noindent\textbf{Conventions.}
We work in the smooth category. All manifolds are oriented. Boundaries of oriented manifolds are oriented using the ``outward normal first'' convention. 

\vspace{1ex}

\noindent\textbf{Acknowledgments.}
I wish to thank Emmanuel Wagner for motivating conversations and Jae Choon Cha for an interesting suggestion.

\section{Definitions and main statements}

\begin{figure}[htb] 
\begin{center}
\begin{tikzpicture} [scale=0.6]
\begin{scope} [xshift=-5cm]
 \draw[fill=gray!40,color=gray!40] (0,0) circle (2 and 1);
 \draw[fill=gray!70,color=gray!70] (-0.5,0) -- (-0.5,2) -- (0.5,2) -- (0.5,0);
 \draw[fill=gray!70,color=gray!70] (0.5,-2) -- (0.5,-0.99) -- (0,-1) -- (-0.5,-0.99) -- (-0.5,-2);
 \foreach \x in {-0.5,0.5} {
 \draw (\x,0) -- (\x,2);
 \draw (\x,-0.98) -- (\x,-2);
 \draw[dashed] (\x,-0.95) -- (\x,-0.05);}
 \draw[color=gray!40] (-2,0) arc (-180:0:2 and 1);
\end{scope}
\begin{scope} [scale=0.8]
 \tour
 \draw (-0.3,-1.3) -- (-0.5,-0.2);
 \draw (-0.4,-0.7) node[right] {$i$};
 \draw[rotate=240] (-1.43,-1.4) -- (1.43,-1.4);
 \draw (-1,1) node[right] {$b$};
\end{scope}
 \draw (-2.5,-2.8) node {ribbon};
\begin{scope}[xshift=6.5cm,yscale=0.5,xscale=0.7]
 \draw[fill=gray!40] (-2,-3) arc (-90:90:3);
 \draw[fill=white] (2,3) arc (90:270:3);
 \draw[fill=gray!60] (2,3) arc (90:270:3);
 \draw[fill=white] (-2,-3) arc (-90:0:3) -- (-2,0);
 \draw[gray!40,fill=gray!40] (-2,-3) arc (-90:0:3) -- (-2,0);
 \draw (-2,-3) arc (-90:0:3);
 \draw[gray!40,line width=2pt] (-2,0) -- (-1.01,0);
\end{scope}
\begin{scope} [xshift=11cm,scale=0.8]
 \tour
 \draw (-1.5,-1.3) -- (-0.5,-0.5);
 \draw (-1.5,1.3) -- (-0.5,0.5);
\end{scope}
 \draw (8.7,-2.8) node {clasp};
\end{tikzpicture}
\end{center}
\caption{Lines of double points and their preimages} 
\label{figribbonclasp}
\end{figure}
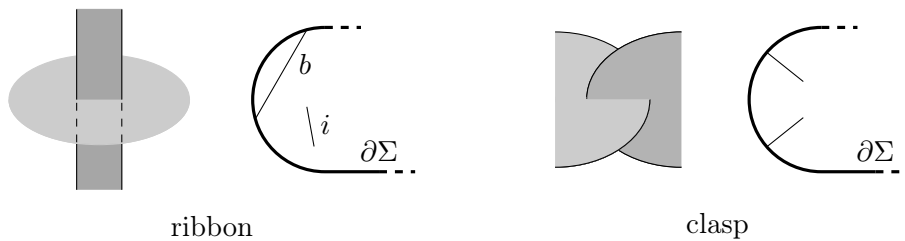

If $\Sigma$ is a compact surface immersed in $S^3$, the self-intersections of $\Sigma$ are lines of double points, which possibly intersect along triple points. The lines of double points are circles or intervals; the latter are of two kinds: ribbons and clasps (see Figure~\ref{figribbonclasp}). A {\em ribbon} is a line of double points whose preimages by the immersion are a {\em $b$--line} properly immersed in $\Sigma$ and an {\em $i$--line} immersed in the interior of $\Sigma$. A {\em clasp} is a line of double points that is not a ribbon. 
When three ribbons meet at a triple point, there are again two possibilities. We say that the triple point is {\em borromean} if its three preimages are intersections of a $b$--line and an $i$--line (see Figures~\ref{figtriplepoints} and~\ref{fig(non)Borromean}). We will also consider surfaces with {\em branch points}, namely points that have a neighborhood as represented in Figure~\ref{figbranchedpt}.

\begin{figure}[htb]
\begin{center}
\begin{tikzpicture} [scale=0.6]
 \draw (0,0) circle (2);
 \draw[red] (-1.43,-1.4) -- (1.43,-1.4);
 \draw[green] (0,-1.8) -- (0,-1);
 \draw[rotate=120,green] (-1.43,-1.4) -- (1.43,-1.4);
 \draw[rotate=120,blue] (0,-1.8) -- (0,-1);
 \draw[rotate=240,blue] (-1.43,-1.4) -- (1.43,-1.4);
 \draw[rotate=240,red] (0,-1.8) -- (0,-1);
 \draw (0,-1.4) node[above right] {$\scriptstyle{p_1}$};
 \draw (-1,0.2) node {$\scriptstyle{p_2}$};
 \draw (1,0.2) node {$\scriptstyle{p_3}$};
 \draw (0,-2.8) node {borromean};
 \begin{scope} [xshift=6cm]
  \draw (0,0) circle (2);
  \draw[red] (-1.43,-1.4) -- (1.43,-1.4);
  \draw [blue] (0,-1.8) -- (0,-1);
  \draw[rotate=240,blue] (-1.43,-1.4) -- (1.43,-1.4);
  \draw[green] (-1.43,1.4) -- (1.43,1.4);
  \draw[red] (-0.3,0) -- (0.5,0); 
  \draw[green] (0.1,-0.4) -- (0.1,0.4);
  \draw (0,-1.4) node[above right] {$\scriptstyle{p_1}$};
  \draw (-0.6,0.9) node {$\scriptstyle{p_2}$};
  \draw (0.6,0.3) node {$\scriptstyle{p_3}$};
  \draw (0,-2.8) node {non borromean};
 \end{scope}
\end{tikzpicture}
\caption{Triple points on a disk\vspace{0.8ex}\\{\footnotesize The picture represents the singular set of the disk on its preimage.\\ The points $p_1$, $p_2$ and $p_3$ are the three preimages of a triple point $p$.}} \label{figtriplepoints}
\end{center}
\end{figure}

\begin{figure}[htb] 
\begin{center}
\begin{tikzpicture} [scale=0.25]
\begin{scope}
 \draw[blue!10,fill=blue!10] (-6.5,-2.1) -- (-6.5,3) .. controls +(0,5) and +(-5,0) .. (0,8.5) .. controls +(5,0) and +(0,5) .. (6.5,3) -- (6.5,-2.1) .. controls +(0,-5) and +(5,0) .. (0,-8.5) .. controls +(-5,0) and +(0,-5) .. (-6.5,-2.1); 
 \draw[green,fill=green!20] (0,-6) .. controls +(2,-2) and +(0,-4) .. (3,-3) .. controls +(0,5) and +(3,-3) .. (0,6); 
 \draw[red!30,fill=red!30] (-6.5,2) .. controls +(-3,0) and +(0,2) .. (-10,0) .. controls +(0,-2) and +(-2,0) .. (-6.5,-2) -- (2,-2) -- (0,0) -- (-6.5,0) -- (-6.5,2);
 \draw[red!30,fill=red!30] (6.5,2) .. controls +(3,0) and +(0,2) .. (10,0) .. controls +(0,-2) and +(2,0) .. (6.5,-2) --(3,-2) -- (3,0) -- (6.5,0);
 \draw[green,dashed] (0,-6) .. controls +(-3,3) and +(0,-5) .. (-3,3) .. controls +(0,4) and +(-2,2) .. (0,6); 
 \draw[red,dashed] (-6,2) -- (6,2) (2.2,-2) -- (2.5,-2);
 \draw[blue,dashed] (-6.5,-1.9) -- (-6.5,0) (6.5,0) -- (6.5,-1.9);
 \draw[red] (-7.1,2) .. controls +(-2,0) and +(0,2) .. (-10,0) .. controls +(0,-2) and +(-2,0) .. (-6.5,-2) -- (2,-2) (3.5,-2) -- (6.5,-2) .. controls +(2,0) and +(0,-2) .. (10,0) .. controls +(0,2) and +(2,0) .. (7.1,2);
 \draw[blue] (-6.5,-2.1) .. controls +(0,-5) and +(-5,0) .. (0,-8.5) ..controls +(5,0) and +(0,-5) .. (6.5,-2.1) (-6.5,0) -- (-6.5,3) .. controls +(0,5) and +(-5,0) .. (0,8.5) .. controls +(5,0) and +(0,5) .. (6.5,3) -- (6.5,0);
 \draw[green] (0,-6) .. controls +(2,-2) and +(0,-4) .. (3,-3) .. controls +(0,5) and +(3,-3) .. (0,6); 
 \draw[gray,densely dashed,thick] (2,-2) -- (-2,2) (0,-6) -- (0,6) (-6.5,0) -- (6.5,0);
\end{scope}
\begin{scope} [xshift=25cm]
 \draw[green,fill=green!20] (0,-8) .. controls +(-4,4) and +(0,-6) .. (-3,3) .. controls +(0,5) and +(-3,3) .. (0,8); 
 \draw[blue!10,fill=blue!10] (-6.5,-2.1) -- (-6.5,3) .. controls +(0,3) and +(-5,0) .. (0,6.5) .. controls +(5,0) and +(0,3) .. (6.5,3) -- (6.5,-2.1) .. controls +(0,-3) and +(5,0) .. (0,-6.5) .. controls +(-5,0) and +(0,-3) .. (-6.5,-2.1); 
 \draw[green,fill=green!20] (0,-8) .. controls +(3,-3) and +(0,-5) .. (3,-3) .. controls +(0,6) and +(4,-4) .. (0,8); 
 \draw[red!30,fill=red!30] (-6.5,2) .. controls +(-3,0) and +(0,2) .. (-10,0) .. controls +(0,-2) and +(-2,0) .. (-6.5,-2) -- (2,-2) -- (0,0) -- (-6.5,0) -- (-6.5,2);
 \draw[red!30,fill=red!30] (6.5,2) .. controls +(3,0) and +(0,2) .. (10,0) .. controls +(0,-2) and +(2,0) .. (6.5,-2) --(3,-2) -- (3,0) -- (6.5,0);
 \draw[green,dashed] (0,-8) .. controls +(-4,4) and +(0,-6) .. (-3,3) .. controls +(0,5) and +(-3,3) .. (0,8); 
 \draw[red,dashed] (-6,2) -- (6,2) (2.2,-2) -- (2.5,-2);
 \draw[blue,dashed] (-6.5,-1.9) -- (-6.5,0) (6.5,0) -- (6.5,-1.9) (0.1,-6.5) ..controls +(5,0) and +(0,-3) .. (6.5,-2.1) (0,6.5) .. controls +(5,0) and +(0,3) .. (6.5,3);
 \draw[red] (-7.1,2) .. controls +(-2,0) and +(0,2) .. (-10,0) .. controls +(0,-2) and +(-2,0) .. (-6.5,-2) -- (2,-2) (3.5,-2) -- (6.5,-2) .. controls +(2,0) and +(0,-2) .. (10,0) .. controls +(0,2) and +(2,0) .. (7.1,2);
 \draw[blue] (-6.5,-2.1) .. controls +(0,-3) and +(-5,0) .. (0,-6.5) (3,-6.28) ..controls +(3,0.55) and +(0,-1.9) .. (6.5,-2.1) (-6.5,0) -- (-6.5,3) .. controls +(0,3) and +(-5,0) .. (0,6.5) (1.5,6.48) .. controls +(3.5,-0.1) and +(0,2.85) .. (6.5,3) -- (6.5,0);
 \draw[green] (0,-8) .. controls +(3,-3) and +(0,-5) .. (3,-3) .. controls +(0,6) and +(4,-4) .. (0,8); 
 \draw[gray,densely dashed,thick] (2,-2) -- (-2,2) (0,-6.5) -- (0,6.5) (-6.5,0) -- (6.5,0);
\end{scope}
\end{tikzpicture}
\end{center}
\caption{Borromean and non borromean triple points\vspace{0.8ex}\\{\footnotesize On the left hand side (resp. right hand side), a borromean link (resp. a trivial link) bounds a disks complex with three ribbons and one borromean (resp. non borromean) triple point}} \label{fig(non)Borromean}
\end{figure} 

\begin{samepage}
A compact surface is:
\begin{itemize}
 \item {\em normal singular} if it is immersed in $S^3$ except at a finite number of branch points,
 \item {\em ribbon} if it is immersed in $S^3$ with no clasp and no triple point, 
 \item {\em $T$--ribbon} if it is immersed in $S^3$ with no clasp and no non-borromean triple point, 
 \item {\em slice} if it is smoothly properly embedded in $B^4$.
\end{itemize}
\end{samepage}

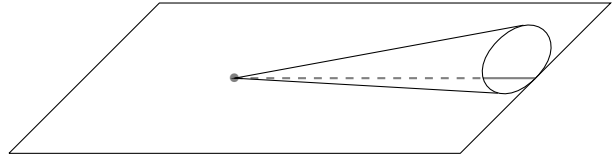
\begin{figure}[htb]
\begin{center}
\begin{tikzpicture} [xscale=0.8]
\begin{scope}
 \draw (-3,-1) -- (3,-1) -- (5,1) -- (-1,1) -- (-3,-1);
 \draw[thick,gray,dashed] (0,0) node {$\scriptstyle{\bullet}$} -- (3.3,0);
 \draw[thick,gray] (3.3,0) -- (4,0);
 \draw (4,0) .. controls +(0.5,0.5) and +(0.5,0.5) .. (3.5,0.5) .. controls +(-0.5,-0.5) and +(-0.5,-0.5) .. (4,0);
 \draw (3.5,-0.2) -- (0,0) -- (3.84,0.705);
 \draw (0,-1.6) node {embedded surface};
\end{scope}
\begin{scope} [xshift=9cm]
 \draw (-3,-1) -- (2.8,-1.2) -- (5.3,1.3) -- (-1,1) -- (-3,-1);
 \draw[thick,gray] (3.5,-0.5) -- (0,0) node {$\scriptstyle{\bullet}$} -- (4.5,0.5);
 \draw (0,-1.6) node {preimage};
\end{scope}
\end{tikzpicture}
\caption{A branched point} \label{figbranchedpt}
\end{center}
\end{figure}

Beyond ribbons and clasps, the different types of double point lines that appear on a normal singular surface are closed lines, namely {\em circles}, or intervals, called {\em branched ribbons} if one endpoint is branched and {\em branched circles} if the two endpoints are branched, see Figure~\ref{figsingularities}, where the preimages are drawn. Like for a ribbon, the two preimages of a branched ribbon are naturally divided into a {\em $b$--line} that contains a boundary point and an {\em $i$--line} that doesn't. For a (branched) circle, one may call $b$--line one preimage and $i$--line the other. A normal singular surface with such namings assigned to the preimages of each (branched) circle is said to be {\em marked}. A triple point on a normal singular surface is {\em borromean} if its three preimages are intersections of a $b$--line and an $i$--line.

\begin{figure}[htb]
\begin{center}
\begin{tikzpicture} [scale=0.5]
 \begin{scope}
 \tour
 \draw (-0.3,-1.3) -- (-0.5,-0.2);
 \draw (-0.4,-0.7) node[right] {$i$};
 \draw[rotate=240] (-1.43,-1.4) -- (1.43,-1.4);
 \draw (-1,1) node[right] {$b$};
 \draw (0.2,-2.8) node {ribbon};
 \end{scope}
 \begin{scope} [xshift=7cm] 
 \tour
 \draw (-0.3,-1.3) -- (0,0.2) -- (-2,0);
 \draw (-0.2,-0.5) node[right] {$i$};
 \draw (-1,0) node[above] {$b$};
 \draw (0.5,-2.8) node {branched ribbon};
 \end{scope}
 \begin{scope} [xshift=14cm] 
 \tour
 \draw (-0.2,-1.2) circle (0.4);
 \draw (-0.8,0.3) circle (0.5);
 \draw (0.3,-2.8) node {circle};
 \end{scope}
 \begin{scope} [xshift=21cm] 
 \tour
 \draw (-0.4,-1.2) .. controls +(0.7,1) and +(0.7,-1) .. (-0.4,1) .. controls +(-0.7,-1) and +(-0.7,1) .. (-0.4,-1.2);
 \draw (0.5,-2.8) node {branched circle};
 \end{scope}
\end{tikzpicture}
\caption{Double points lines on a normal singular surface (preimages)} \label{figsingularities}
\end{center}
\end{figure}
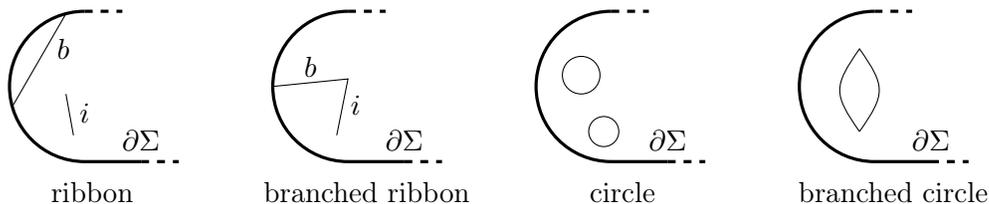

Given a link $L=L_1\sqcup\dots\sqcup L_n$ in $S^3$, a {\em complex for $L$} is a union of compact surfaces $\Sigma=\cup_{1\leq i\leq n}\Sigma_i$ such that $\partial\Sigma_i=L_i$ for all $i$; it is a {\em disks complex} if each $\Sigma_i$ is a disk. The notions defined above for compact surfaces extend to complexes. 
The {\em genus} of such a complex is the sum of the genera of its components. The {\em slice genus $g_s(L)$} (resp. {\em ribbon genus $g_r(L)$}) of a link $L$ is the minimal genus of a slice (resp. ribbon) complex for $L$. Note that $g_s(L)\leq g_r(L)$. These invariants are well-defined (finite) for {\em algebraically split links}, namely links whose components have trivial pairwise linking numbers. The following result generalizes the characterization of slice knots by Kawauchi--Shibuya--Suzuki \cite[Corollary~6.7]{KSS}.

\begin{theo}[Corollary~\ref{corCharSliceGenus}]
 The slice genus of an algebraically split link $L$ equals the minimal genus of a marked normal singular complex for $L$ with no clasp and no borromean triple point.
\end{theo}

The {\em $T$--genus} $T_s(L)$ (resp. {\em ribbon $T$--genus $T_r(L)$}) is the minimal number of borromean triple points on a marked normal singular disks complex (resp. $T$--ribbon disks complex) with no clasp for $L$ (the subscripts $s$ and $r$ stand for slice and ribbon, indeed $T_s$ is related to slice invariants and $T_r$ to ribbon invariants); obviously $T_s(L)\leq T_r(L)$. Note that these numbers may be undefined. Kaplan \cite{Kaplan} proved that any algebraically split link bounds a $T$--ribbon disks complex. In particular, the (ribbon) $T$--genus of an algebraically split link is well-defined. It follows that any of the four invariants $g_s,g_r,T_s,T_r$ is well-defined on a link if and only if this link is algebraically split (see for instance Corollary~\ref{corgenusTgenus}).

In \cite{MS}, Murakami and Sugishita proved that the $T$--genus of a knot is an upper bound for the slice genus. We generalize this result to algebraically split links and prove the ribbon counterpart of it.
\begin{theo}[Corollary~\ref{corgenusTgenus}]
 For any algebraically split link $L$, $g_s(L)\leq T_s(L)$ and $g_r(L)\leq T_r(L)$. 
\end{theo}

The proof we give relies on the expression of the $T$--genus in terms of $\Delta$--distance, which is the distance on the set of links defined by the $\Delta$--move (see Figure~\ref{figdeltamove}). 
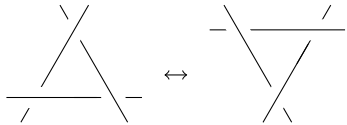
\begin{figure}[htb] 
\begin{center}
\begin{tikzpicture} [scale=0.3]
\newcommand{\DM}{
\draw (-3,-1) -- (3,-1);
\draw[color=white,line width=8pt,rotate=120] (-3,-1) -- (0,-1);
\draw[rotate=120] (-3,-1) -- (3,-1);
\draw[color=white,line width=8pt,rotate=240] (-3,-1) -- (0,-1);
\draw[rotate=240] (-3,-1) -- (3,-1);
\draw[color=white,line width=8pt] (-3,-1) -- (0,-1);
\draw (-3,-1) -- (1,-1);}
\DM
\draw[<->] (4,0) -- (5,0);
\begin{scope} [xshift=9cm,yshift=1cm,rotate=60]
\DM
\end{scope}
\end{tikzpicture}
\end{center}
\caption{$\Delta$--move} \label{figdeltamove}
\end{figure}
The {\em $\Delta$--slicing number $\sd(L)$} (resp. {\em $\Delta$--ribbonning number $\rd(L)$}) is the minimal number of $\Delta$--moves necessary to change $L$ into a slice (resp. ribbon) link. Note that these are well-defined if and only if $L$ is algebraically split (see Theorem~\ref{thMuNa}). Kawauchi, Murakami and Sugishita \cite{KMS} proved that the $T$--genus of a knot equals its $\Delta$--slicing number. We generalize this as follows. 
\begin{theo}[Theorem~\ref{thDeltaDistance}]
 For any algebraically split link $L$, $T_s(L)=\sd(L)$ and $T_r(L)=\rd(L)$.
\end{theo}

It is a natural generalization of the ribbon--slice question to ask wether the $T$--genus and the ribbon $T$--genus always coincide. The above result shows that it is an equivalent question.

\begin{corollary}
 The slice knots are all ribbon if and only if $T_s(K)=T_r(K)$ for any knot $K$. The same holds for algebraically split links with a given number of components.
\end{corollary}

In \cite{MS}, Murakami and Sugishita proved that the $T$--genus of knots is a concordance invariant whose mod--$2$ reduction is the Arf invariant --- we generalize this to algebraically split links. They deduced the value of the $T$--genus for several knots satisfying $T_s-g_s=0,1$. The question arises then to know if this difference can be arbitrarily large. In Section~\ref{secex}, we provide a family of non-split links $B_n$ with $g_s(B_n)=1$ and $T_s(B_n)=n$. These links are constructed by cabling one component of the borromean link. 
Nevertheless, this is not fully satisfying in that it is based on the augmentation of the number of components of the link. To get the answer in the setting of knots, we will compare the $T$--genus with the $4$--dimensional positive clasp number. The {\em $4$--dimensional clasp number} $c_4(L)$ of a link $L$ is the smallest number of transverse double points on an immersed disks complex in $B^4$ bounded by $L$. Similarly define the {\em $4$--dimensional positive/negative clasp number} $c_4^+(L)$/$c_4^-(L)$ by counting the positive/negative double points. We also consider here a balanced version of this invariant, which is the most natural in the comparison with the $T$--genus. The {\em balanced $4$--dimensional clasp number} $c_4^b(L)$ of a link $L$ is the smallest number of positive transverse double points on an immersed disks complex $\Sigma$ in $B^4$, bounded by $L$, with trivial self-intersection number ({\em ie} $\Sigma$ has the same number of positive and negative transverse double points). We have the following immediate inequalities (note that a positive or negative transverse double point can be added to an immersed surface in $B^4$ without modifying its boundary).
$$c_4^+(L),c_4^-(L)\leq c_4^b(L)\leq c_4(L)\leq 2c_4^b(L)$$

Note that, among all the invariants introduced at that point, only $c_4^\pm$ depends on the orientation, in the case of a link with at least two components.

It is well known that the $4$--dimensional clasp number is an upper bound for the slice genus --- a transverse double point can be smoothed at the cost of adding one to the genus of the surface. For knots, this can be improved as follows.

\begin{lem}[Lemma~\ref{lemmagscb}]
 For any knot $K$, $g_s(K)\leq c_4^b(K)$.
\end{lem}

We will see in Section~\ref{secex} that this lemma does not hold for algebraically split links. In contrast, we have the following.

\begin{theo}[Corollary~\ref{corTgenusclaspnb}]
 For any algebraically split link $L$, $c_4^b(L)\leq T_s(L)$.
\end{theo}

In \cite{DaSca}, Daemi and Scaduto minored the difference $c_4^+(K)-g_s(K)$ for the connected sums of copies of the knot $7_4$.

\begin{theorem}[Daemi--Scaduto]
 For any positive integer $n$, $c_4^+(\sharp^n7_4)-g_s(\sharp^n7_4)\geq n/5$.
\end{theorem}

\begin{corollary}
 The difference between the $T$--genus and the slice genus, for knots in $S^3$, can be arbitrarily large.
\end{corollary}

This raises the question of what can be the difference between the $T$--genus and the balanced 4--dimensional clasp number. In \cite{Miller}, Miller proves the existence of knots with arbitrarily large slice genus and trivial positive and negative 4--dimensional clasp number, which implies that the difference $T_s-c_4^{\pm}$ can be arbitrarily large.

\begin{question}
 Can the difference $T_s-c_4^b$ be arbitrarily large  for knots?
\end{question}

In Section~\ref{secex}, we propose a family of knots that could realize such an unbounded difference, namely the family of twist knots. The 4--dimensional clasp number equals $1$ for all non-slice twist knots; I would conjecture that their $T$--genus raises linearly with the number of twists.

\section{Surfaces, cobordisms and projections}

In this section, we present some manipulations on surfaces and cobordisms and we deduce a comparison of the $T$--genus with the slice genus and the balanced $4$--dimensional clasp number, and a characterization of the slice genus.

On the preimage of a marked normal singular compact surface, a preimage of a triple point is a {\em triple point of type ($b$--$i$)} if it is the intersection of a $b$--line and an $i$--line.

\begin{proposition} \label{propborrotoclasp}
 Let $\Sigma$ be a marked normal singular compact surface in $S^3$ with $b$ borromean triple points and no clasp. Then $\Sigma$ is the radial projection of a properly immersed surface in $B^4$ with $b$ positive and $b$ negative double points.
\end{proposition}
\begin{proof}
 Let $\tilde{\Sigma}$ be a compact surface and let $\iota:\tilde{\Sigma}\to S^3$ be a map which is an immersion except at a finite number of branch points, such that $\iota(\tilde\Sigma)=\Sigma$. We will define a radius function on $\tilde\Sigma$ in order to immerse it in $B^4$. Let $r:\tilde\Sigma\to(0,1]$ be a smooth function that sends:
 \begin{itemize}
  \item $\partial\tilde\Sigma$ onto $1$,
  \item branch points and triple points of type ($b$--$i$) onto $\frac12$,
  \item other points in an $i$--line to $(0,\frac12)$,
  \item other points in a $b$--line to $(\frac12,1)$,
 \end{itemize}
 see Figure~\ref{figmapr}. 
 Note that, within the three preimages of a non-borromean triple point, only one has type ($b$--$i$), so that only one is sent onto $\frac12$. The map from $\tilde\Sigma$ to $B^4=\frac{S^3\times[0,1]}{(x,0)\sim(y,0)}$ given by $x\mapsto\left(x,r(x)\right)$ is an embedding except at borromean triple points which are still triple points. It remains to modify the radius function around these borromean triple points.
 
\begin{figure}[htb]
\begin{center}
\begin{tikzpicture} [thick,scale=0.8]
 \draw[red] (0,0) circle (2) (2,1.2) node {$1$};
 \draw[blue] (-1.43,-1.4) -- (1.43,-1.4) (-1.43,1.4) -- (1.43,1.4) (2,0) -- (0.6,-0.2) (0,1.4) node[below right] {$\left(\frac12,1\right)$};
 \draw[rotate=240,blue] (-1.43,-1.4) -- (1.43,-1.4);
 \draw [gray] (0,-1.8) -- (0,-1) (-0.7,0) -- (0.1,0) (-0.3,-0.4) -- (-0.3,0.4) (0.6,-0.2) -- (1.2,0.5) (-1,0) node[below] {$\left(0,\frac12\right)$};
 \draw[vert] (0,-1.4) node {$\bullet$} (0.6,-0.2) node {$\bullet$} node[below left] {$\frac12$};
\end{tikzpicture}
\caption{The radius function} \label{figmapr}
\end{center}
\end{figure}
 
\begin{figure}[htb]
\begin{center}
\begin{tikzpicture}
 \foreach \x in {0,5.5,9,12.5} {\draw[xshift=\x cm] (0,2) node[left] {$0$} -- (0,0) node[left] {$1$} -- (2,0) (0,1) node[left] {$r$};}
 \draw (1,0) node[below] {$\tilde\ell_{ij}$};
 \draw[thick] (2,0) arc (7:173:1) (2,1.76) arc (-7:-173:1) (1,0.86) node {$\scriptstyle\bullet$};
 \draw[very thick,->] (3,1) -- (4,1);
\begin{scope} [xshift=5.5cm]
 \draw (1,0) node[below] {$\tilde\ell_{12}$};
 \draw[thick,rounded corners=1pt] (2,0) arc (7:70:1) arc (20:160:0.37) arc (110:173:1) (2,1.76) arc (-7:-70:1) arc (-20:-160:0.37) arc (-110:-173:1);
\end{scope}
\begin{scope} [xshift=9cm]
 \draw (1,0) node[below] {$\tilde\ell_{23}$};
 \draw[thick,rounded corners=2pt] (2,0) arc (7:173:1) (2,1.76) arc (-7:-70:1) arc (20:160:0.37) arc (-110:-173:1);
\end{scope}
\begin{scope} [xshift=12.5cm]
 \draw (1,0) node[below] {$\tilde\ell_{31}$};
 \draw[thick,rounded corners=2pt] (2,0) arc (7:70:1) arc (-20:-160:0.37) arc (110:173:1) (2,1.76) arc (-7:-173:1);
\end{scope}
\end{tikzpicture}
\caption{Modification of the radius function around a triple point\vspace{0.8ex}\\{\footnotesize The line $\ell_{ij}$ is the ribbon line on $\Sigma$ containing $p$\\whose preimage $\tilde\ell_{ij}$ is made of an $i$--line containing $p_i$ and a $b$--line containing $p_j$.}} \label{figmodifradius}
\end{center}
\end{figure}
 
 Take a smooth function $f:D^2\to[0,\varepsilon]$, where $\varepsilon>0$, such that $f=0$ on a collar neighborhood of $S^1$, $f(0)=\varepsilon$ and~$0$ is the only critical point of $f$. Fix a borromean triple point $p$ on $\Sigma$ and write $p_1,p_2,p_3$ its preimages, chosing the indices so that the $i$--line containing $p_1$ has the same image as the $b$--line containing~$p_2$. Now consider small disk neighborhoods of $p_1$ and $p_2$ and define a new radius function $r'$ by $r'=r-f$ in the neighborhood of $p_1$, $r'=r+f$ in the neighborhood of $p_2$ and $r'=r$ elsewhere. This desingularizes the triple point but adds two double points of opposite signs with preimages on the $i$--line containing $p_1$ and on the $b$--line containing $p_2$, see Figure~\ref{figmodifradius}. If the neighborhoods and $\varepsilon$ are small enough, no other singularity is created. Performing a similar modification around each borromean triple point, we finally get the required immersion. 
\end{proof}

\begin{corollary} \label{corTgenusclaspnb}
 For any algebraically split link $L$, $c_4^b(L)\leq T_s(L)$.
\end{corollary}

For knots, this corollary and the following lemma give a relation between the slice genus and the $T$--genus. Neverthless, the lemma does not hold for algebraically split links and we will give alternative proofs of this relation.

\begin{lemma} \label{lemmagscb}
 For any knot $K$, $g_s(K)\leq c_4^b(K)$.
\end{lemma}
\begin{proof}
 Take a disk $\Sigma$ properly immersed in $B^4$ with $\partial\Sigma=K$ that realizes $c_4^b(K)$ (in particular $\Sigma$ has the same number of positive and negative double points). We will desingularize $\Sigma$ by tubing once for each pair of double points with opposite signs. Fix a pair $(p_+,p_-)$ of respectively a positive and a negative self-intersection point of $\Sigma$. Take a path $\gamma$ on $\Sigma$ joining $p_+$ and $p_-$, whose interior does not meet the singularities of $\Sigma$. Inside a neighborhood of $\gamma$ in $B^4$, one can find a solid tube $C=\gamma\times D^2$ where $\gamma\times\{0\}=\gamma$, such that $C$ meets the leaf of $\Sigma$ transverse to $\gamma$ around $p_+$ (resp. $p_-$) along $p_+\times D^2$ (resp. $p_-\times D^2$), and $C$ meets the leaf of $\Sigma$ containing $\gamma$ along $\gamma$. Now remove from $\Sigma$ the interior of the disks $p_\pm\times D^2$ and reglue $\gamma\times S^1$ instead. The sign condition ensures that the new surface is again oriented.
\end{proof}

Proposition~\ref{propborrotoclasp} says in particular that the boundary of a marked normal singular genus--$g$ complex bounds a genus--$g$ slice complex. We now prove the converse in order to get a three-dimensional characterization of the slice genus.

\begin{theorem} \label{thprojection}
 Let $S$ be a properly embedded compact surface in $B^4$. Then $S$ is isotopic, via an isotopy of properly embedded surfaces, to a surface whose radial projection in $S^3$ is a marked normal singular compact surface with no clasp and no borromean triple point.
\end{theorem}
\begin{proof}
 Up to isotopy, we can assume that the radius function on $B^4$ restricts on $S$ to a Morse function whose index $\ell$ critical points take the value $\varepsilon_\ell$ with $0<\varepsilon_0<\varepsilon_1<\varepsilon_2<1$. Hence $S$ can be reorganized as a PL surface with:
 \begin{itemize}
  \item at $r=\varepsilon_0$, a family $\Delta$ of disjoint disks,
  \item at $r\in(\varepsilon_0,\varepsilon_1)$, the boundary $\partial\Delta$,
  \item at $r=\varepsilon_1$, $\partial\Delta\cup B$, where $B$ a disjoint union of embedded bands $[0,1]\times[0,1]$ that meet $\partial\Delta$ exactly along $[0,1]\times\{0,1\}$,
  \item at $r\in(\varepsilon_1,\varepsilon_2)$, the split union of $L=\partial S$ with a trivial link,
  \item at $r=\varepsilon_2$, the disjoint union of $L$ with a disjoint union $D$ of disks bounded by the above trivial link,
  \item at $r\in(\varepsilon_2,1]$, the link $L$.
 \end{itemize}
 Projecting radially $S$ on $S^3$, we obtain a complex $\Sigma=\Sigma'\cup_{\partial D} D$, where $\Sigma'$ is a ribbon complex --- projection of $\Delta\cup B$ --- and $D$ is a family of disjoint embedded disks, disjoint from $L$, that we can assume transverse to $\Sigma'$. Denote $\tilde\Sigma=\tilde\Sigma'\cup\tilde D$ the preimage of $\Sigma$ with the corresponding decomposition.
 
 Let us make a few useful remarks. Since $D$ is embedded, a double point $p\in\partial D$ is a double point of $\Sigma'$. Since $\Sigma'$ is a ribbon complex and $D$ is embedded, all branch points of $\Sigma$ lie on $\partial D$. By definition, branch points are simple. 
 
 We now analyse the lines of double points on $\Sigma$ and their preimages. Let $\gamma\subset\Sigma$ be the closure of a line of double points. Note that it cannot have an endpoint in $\Int(D)$ since $D$ is disjoint from $L$. 
 \paragraph{\underline{Case 1}:} $\gamma$ does not contain any branch point.
 Let $\tilde\gamma$ be one of its two preimages. If $\tilde\gamma\subset\tilde D$, then $\gamma$ is a circle and we mark $\tilde\gamma$ as a $b$--line. 
 If $\tilde\gamma$ meets both $\tilde D$ and $\tilde\Sigma'$, then it contains an interval properly embedded in $\tilde D$, see Figure~\ref{figNoBranchPoint}. It has to be continued on both sides by $b$--lines of $\tilde\Sigma'$. These end either on $\partial\tilde\Sigma$ or on $\partial\tilde D$. In the latter case, they are continued by an interval properly embedded in $\tilde D$. Iterating, we see that $\tilde\gamma$ is either an interval with two endpoints on $\partial\tilde\Sigma$ or a circle. In both cases, we mark it as a $b$--line. Now assume $\tilde\gamma\subset\tilde\Sigma'$. If it is an interval, it has no endpoint on $\partial\tilde D$, so that it is a $b$--line or an $i$--line of $\tilde\Sigma'$. If it is a circle, either the other preimage of $\gamma$ meets $\tilde D$, in which case we mark $\tilde\gamma$ as an $i$--line, or the two preimages of $\gamma$ are contained in $\tilde\Sigma'$ and we assign them different markings.
 \begin{figure}[htb]
 \begin{center}
 \begin{tikzpicture} [scale=0.8]
 \begin{scope}
  \draw[very thick,dashed] (4,-2) -- (5,-2) (3,2) -- (4,2);
  \draw[very thick] (3,2) -- (0,2) arc (90:270:2) -- (4,-2) node[above] {$\partial\tilde\Sigma$};
  \draw[fill=gray!30] (0,0.5) circle (0.8) (2,-0.5) circle (0.6);
  \draw (-0.1,0.95) node {$\tilde D$};
  \draw (-0.8,0.5) node {$\scriptstyle{\bullet}$} -- (0.8,0.5) node {$\scriptstyle{\bullet}$};
  \draw (-0.8,0.5) .. controls +(-1,0) and +(0.3,0.3) .. (-1.5,-1.3) node {$\scriptstyle{\bullet}$};
  \draw (0.8,0.5) .. controls +(0.5,0) and +(0,0.5) .. (2,0.1) node {$\scriptstyle{\bullet}$} .. controls +(0,-0.5) and +(-0.5,0) .. (2.6,-0.5) node {$\scriptstyle{\bullet}$} .. controls +(0.5,0) and +(0,0.5) .. (3,-2) node {$\scriptstyle{\bullet}$};
 \end{scope}
 \begin{scope} [xshift=10cm]
  \draw[very thick,dashed] (4,-2) -- (5,-2) (3,2) -- (4,2);
  \draw[very thick] (3,2) -- (0,2) arc (90:270:2) -- (4,-2) node[above] {$\partial\tilde\Sigma$};
  \draw[fill=gray!30] (0,0.5) circle (0.8) (2,-0.5) circle (0.6);
  \draw (-0.1,0.95) node {$\tilde D$};
  \draw (-0.8,0.5) node {$\scriptstyle{\bullet}$} -- (0.8,0.5) node {$\scriptstyle{\bullet}$};
  \draw (-0.8,0.5) .. controls +(-2,0) and +(-1.5,-1.5) .. (1.55,-0.9) node {$\scriptstyle{\bullet}$};
  \draw (0.8,0.5) .. controls +(0.5,0) and +(0,0.5) .. (2,0.1) node {$\scriptstyle{\bullet}$} .. controls +(0,-0.3) and +(0.5,0.5) .. (1.55,-0.9);
 \end{scope}
 \end{tikzpicture}
 \caption{Case of no branch point and $\tilde\gamma\cap\partial\tilde D\neq\emptyset$} \label{figNoBranchPoint}
 \end{center}
 \end{figure}

 \paragraph{\underline{Case 2}:} $\gamma$ contains a single branch point $p$.
 Let $\tilde\gamma$ be the whole preimage of $\gamma$ and let $\tilde p$ be the preimage of $p$. Near $p$, $\gamma$ lies in $\Sigma'\cap D$. Let $q$ be the first point of $\partial D$ reached from $p$. The preimages of $q$ are $\tilde q_1\in\partial\tilde D$ and $\tilde q_2\in\Int(\tilde\Sigma')$, see Figure~\ref{figSingleBranchPoint}. The curve $\tilde\gamma$ joins $\tilde p$ to $\tilde q_1$ inside $\tilde D$ and to $\tilde q_2$ inside $\tilde\Sigma'$. It is continued from $\tilde q_1$ by a $b$--line in $\tilde\Sigma'$ and from $\tilde q_2$ by an $i$--line in $\Sigma'$. The $b$--line may end on $\partial\tilde D$, in which case we iterate the argument, or on $\partial\tilde\Sigma$. Finally, the preimage of $\gamma$ containing $\tilde q_1$ is a $b$--line ending on $\partial\tilde\Sigma$ and the preimage containing $\tilde q_2$ is an $i$--line contained in $\tilde\Sigma'$.
 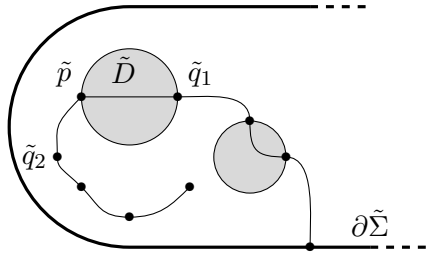
\begin{figure}[htb]
 \begin{center}
 \begin{tikzpicture} [scale=0.8]
  \draw[very thick,dashed] (4,-2) -- (5,-2) (3,2) -- (4,2);
  \draw[very thick] (3,2) -- (0,2) arc (90:270:2) -- (4,-2) node[above] {$\partial\tilde\Sigma$};
  \draw[fill=gray!30] (0,0.5) circle (0.8) (2,-0.5) circle (0.6);
  \draw (-0.1,0.95) node {$\tilde D$};
  \draw (-0.8,0.5) node {$\scriptstyle{\bullet}$} node[above left] {$\tilde{p}$} -- (0.8,0.5) node {$\scriptstyle{\bullet}$} node[above right] {$\tilde{q}_1$};
  \draw (-0.8,0.5) .. controls +(-0.5,-0.5) and +(0,0.5) .. (-1.2,-0.5) node {$\scriptstyle{\bullet}$} node[left] {$\tilde{q}_2$} .. controls +(0,-0.2) and +(-0.3,0.3) .. (-0.8,-1) node {$\scriptstyle{\bullet}$} .. controls +(0.3,-0.3) and +(-0.5,0) .. (0,-1.5) node {$\scriptstyle{\bullet}$} .. controls +(0.5,0) and +(-0.3,-0.3) .. (1,-1) node {$\scriptstyle{\bullet}$};
  \draw (0.8,0.5) .. controls +(0.5,0) and +(0,0.5) .. (2,0.1) node {$\scriptstyle{\bullet}$} .. controls +(0,-0.5) and +(-0.5,0) .. (2.6,-0.5) node {$\scriptstyle{\bullet}$} .. controls +(0.5,0) and +(0,0.5) .. (3,-2) node {$\scriptstyle{\bullet}$};
 \end{tikzpicture}
 \caption{Case of a single branch point $p$} \label{figSingleBranchPoint}
 \end{center}
 \end{figure}
 
 \paragraph{\underline{Case 3}:} $\gamma$ contains two branch points, {\em ie} it is an interval with branch points as endpoints. The two preimages of $\gamma$ are intervals in $\tilde\Sigma$ with the same endpoints. Analyzing the situation from a branch point as in the previous case, we see that one preimage of $\gamma$ is contained in $\tilde\Sigma'$, while the other meets $\tilde D$. We mark the first one as an $i$--line and the second one as a $b$--line.
 
 Finally $\Sigma$ is a marked normal singular complex. Let $p$ be a triple point of $\Sigma$. Since $\Sigma'$ is a ribbon complex, $p$ must have at least one preimage $\tilde p$ in $\Int(\tilde D)$. It follows from the previous discussion that no $i$--line meets $\Int(\tilde D)$, so that $\tilde p$ is the intersection of two $b$--lines. Hence $p$ is non-borromean.
\end{proof}

\begin{corollary} \label{corCharSliceGenus}
 The slice genus of an algebraically split link $L$ equals the minimal genus of a marked normal singular complex for $L$ with no clasp and no borromean triple point.
\end{corollary}

A {\em cobordism} from a link $L$ to a link $L'$ is a surface $S$ properly embedded in $S^3\times[0,1]$, such that $\partial S=L'\times\{1\}- L\times\{0\}$; the links $L$ and $L'$ are said to be {\em cobordant}. The cobordism $S$ is {\em strict} if it has genus $0$ and distinct components of $L$ belong to distinct components of $S$; note that the relation induced on links is not symmetric. A cobordism is a {\em concordance} if $S$ is a disjoint union of annuli and each annulus has one boundary component in $S^3\times\{0\}$ and the other in $S^3\times\{1\}$; the links $L$ and $L'$ are then said to be {\em concordant}. 

\begin{theorem} \label{thconcordance}
 Let $\Sigma$ be a marked normal singular compact surface in $S^3$ with $b$ borromean triple points and no clasp. Let $S$ be a compact surface properly embedded in $S^3\times[0,1]$ such that $\partial S\cap\left(S^3\times\{1\}\right)=-\partial\Sigma\times\{1\}$. Then, up to an isotopy of $S$ fixing the boundary, the image of $S\cup\left(\Sigma\times\{1\}\right)$ by the projection $S^3\times[0,1]\twoheadrightarrow S^3$ is again a marked normal singular compact surface with $b$ borromean triple points and no clasp. 
\end{theorem}
\begin{proof}
 The proof is essentially the same as for Theorem~\ref{thprojection}. This proof still works if the surface~$S$, at $r=\varepsilon_0$, is not a disjoint union of disks but a marked normal singular complex $S$. One can ask that, in the projected complex $\Sigma$, the bands added at $r=\varepsilon_1$ avoid the singularities of $S$ and the union of disks $D$ avoids the triple points of $S$. The same discussion shows that $\Sigma$ can be marked as a normal singular surface with as many borromean triple points as $S$. Here, we consider $S$ in~$S^3\times[0,1]$ with $0<\varepsilon_2<\varepsilon_1<\varepsilon_0=1$ and we project on $S^3\times\{0\}$. 
\end{proof}

\begin{corollary} \label{corconcordance}
 The $T$--genus of algebraically split links is a concordance invariant.
\end{corollary}

\begin{figure}[htb] 
\begin{center}
\begin{tikzpicture} [scale=0.4]
\begin{scope} [yshift=0.5cm]
 \foreach \t in {0,120,240} {
 \draw[rotate=\t] (-1.77,0.27) arc (135:350:2.5);
 \draw[rotate=\t] (2.46,-1.07) arc (10:105:2.5);}
\end{scope}
 \draw (5.5,0) node {$\sim$};
\begin{scope} [xshift=12cm,scale=0.5]
 \draw (0,-6) .. controls +(-3,3) and +(0,-5) .. (-3,3) .. controls +(0,4) and +(-2,2) .. (0,6); 
 \draw[white,line width=5pt] (-5.9,2) -- (5.9,2);
 \draw (-5.9,2) -- (5.9,2);
 \draw (-6.5,-1.6) -- (-6.5,3) .. controls +(0,5) and +(-5,0) .. (0,8.5) .. controls +(5,0) and +(0,5) .. (6.5,3) -- (6.5,-1.6);
 \draw[white,line width=5pt] (-7.1,2) .. controls +(-2,0) and +(0,2) .. (-10,0) .. controls +(0,-2) and +(-2,0) .. (-6.5,-2) -- (6.5,-2);
 \draw (-7.1,2) .. controls +(-2,0) and +(0,2) .. (-10,0) .. controls +(0,-2) and +(-2,0) .. (-6.5,-2) -- (6.5,-2);
 \draw (7.1,2) .. controls +(2,0) and +(0,2) .. (10,0) .. controls +(0,-2) and +(2,0) .. (6.5,-2);
 \draw (-6.5,-2.4) .. controls +(0,-5) and +(-5,0) .. (0,-8.5) ..controls +(5,0) and +(0,-5) .. (6.5,-2.4);
 \draw[white,line width=5pt] (0,-6) .. controls +(2,-2) and +(0,-4) .. (3,-3) .. controls +(0,5) and +(3,-3) .. (0,6);
 \draw (0,-6) .. controls +(2,-2) and +(0,-4) .. (3,-3) .. controls +(0,5) and +(3,-3) .. (0,6);  
\end{scope}
\begin{scope} [xshift=25cm,scale=0.5]
 \draw[blue!10,fill=blue!10] (-6.5,-2.1) -- (-6.5,3) .. controls +(0,5) and +(-5,0) .. (0,8.5) .. controls +(5,0) and +(0,5) .. (6.5,3) -- (6.5,-2.1) .. controls +(0,-5) and +(5,0) .. (0,-8.5) .. controls +(-5,0) and +(0,-5) .. (-6.5,-2.1); 
 \draw[green,fill=green!20] (0,-6) .. controls +(2,-2) and +(0,-4) .. (3,-3) .. controls +(0,5) and +(3,-3) .. (0,6); 
 \draw[red!30,fill=red!30] (-6.5,2) .. controls +(-3,0) and +(0,2) .. (-10,0) .. controls +(0,-2) and +(-2,0) .. (-6.5,-2) -- (2,-2) -- (0,0) -- (-6.5,0) -- (-6.5,2);
 \draw[red!30,fill=red!30] (6.5,2) .. controls +(3,0) and +(0,2) .. (10,0) .. controls +(0,-2) and +(2,0) .. (6.5,-2) --(3,-2) -- (3,0) -- (6.5,0);
 \draw[vert,dashed] (0,-6) .. controls +(-3,3) and +(0,-5) .. (-3,3) .. controls +(0,4) and +(-2,2) .. (0,6); 
 \draw[red,dashed] (-6,2) -- (6,2) (2.2,-2) -- (2.5,-2);
 \draw[blue,dashed] (-6.5,-1.9) -- (-6.5,0) (6.5,0) -- (6.5,-1.9);
 \draw[red] (-7.1,2) .. controls +(-2,0) and +(0,2) .. (-10,0) .. controls +(0,-2) and +(-2,0) .. (-6.5,-2) -- (2,-2) (3.5,-2) -- (6.5,-2) .. controls +(2,0) and +(0,-2) .. (10,0) .. controls +(0,2) and +(2,0) .. (7.1,2);
 \draw[blue] (-6.5,-2.1) .. controls +(0,-5) and +(-5,0) .. (0,-8.5) ..controls +(5,0) and +(0,-5) .. (6.5,-2.1) (-6.5,0) -- (-6.5,3) .. controls +(0,5) and +(-5,0) .. (0,8.5) .. controls +(5,0) and +(0,5) .. (6.5,3) -- (6.5,0);
 \draw[green] (0,-6) .. controls +(2,-2) and +(0,-4) .. (3,-3) .. controls +(0,5) and +(3,-3) .. (0,6); 
 \draw[gray,densely dashed,thick] (2,-2) -- (-2,2) (0,-6) -- (0,6) (-6.5,0) -- (6.5,0);
\end{scope}
\end{tikzpicture}
\end{center}
\caption{The borromean link and a $T$--ribbon disks complex for it} \label{figDisksBorromean}
\end{figure}
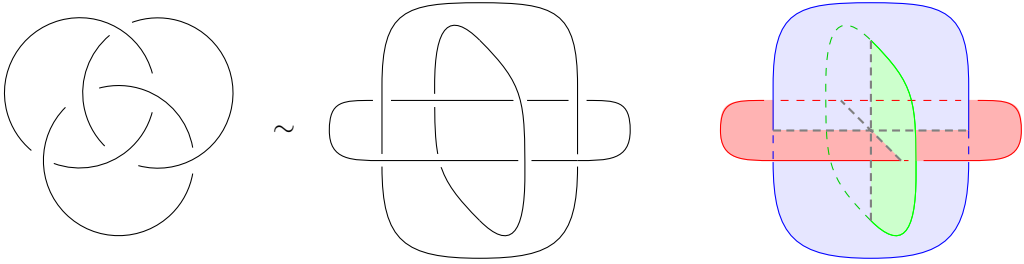 

The next result says that a marked normal singular surface in $S^3$ can be pushed into $S^3\times[0,1]$ in order to isolate the borromean triple points.

\begin{proposition} \label{propcob}
 Let $\Sigma$ be a marked normal singular surface in $S^3$ with $b$ borromean triple points and no clasp. Then there is a surface $S$ properly embedded in $S^3\times[0,1]$ and a disjoint union $\Delta\subset S^3\times\{1\}$ of $b$ $T$--ribbon disks complexes as represented in Figure~\ref{figDisksBorromean} such that $S\cap(S^3\times\{1\})=-\partial\Delta$, $S\cap(S^3\times\{0\})=\partial\Sigma$ and the image of $S\cup\Delta$ by the projection $S^3\times[0,1]\twoheadrightarrow S^3$ is $\Sigma$.
\end{proposition}
\begin{proof}
 Let $\iota:\tilde{\Sigma}\to S^3$ be an immersion except at a finite number of branch points, with image~$\Sigma$. We will define a height function on $\tilde\Sigma$ in order to immerse it in $S^3\times[0,1]$ as a cobordism from $\partial\Sigma$ to a split union of borromean links. 
 
 First define a subsurface $\tilde C_1$ of $\tilde\Sigma$ as the disjoint union of:
 \begin{itemize}
  \item a small disk around each triple point of type ($b$--$b$), namely intersection of two $b$--lines,
  \item for each $b$--line, a small disk around a point of the line between any two consecutive triple points of type ($b$--$i$),
  \item for each closed $b$--line with no triple point, a small disk around any point of the circle,
  \item at each branched point, a small disk meeting the $b$--line along an open interval admitting the branched as an endpoint,
  \item a collar neighborhood of $\partial\Sigma$.
 \end{itemize}
 Set $\tilde\Sigma_1=\tilde\Sigma\setminus\Int(\tilde C_1)$, see Figure~\ref{figpushcobborro}. In restriction to $\tilde\Sigma_1$, $\iota$ is an immersion whose image is a $T$--ribbon genus--$0$ surface such that any ribbon line contains at most one triple point. Now define $\tilde\Sigma_2$ as a neighborhood in $\tilde\Sigma_1$ of all the $i$--lines. Set $\tilde C_2=\tilde\Sigma_1\setminus\Int(\tilde\Sigma_2)$. Then $\tilde\Sigma_2$ is a disjoint union of disks. On some of them, the restriction of $\iota$ to $\tilde\Sigma_2$ has no singularity; define $\tilde C_3$ as their union with a collar neighborhood of the boundary of the others. Set $\tilde\Sigma_3=\tilde\Sigma_2\setminus\Int(\tilde C_3)$. The restriction of $\iota$ to $\tilde\Sigma_3$ has, on each disk, one $b$--line and one $i$--line intersecting once. 
 
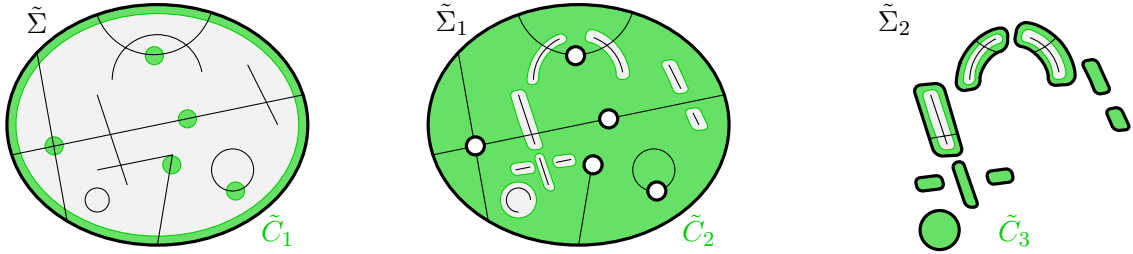
\begin{figure}[htb]
\begin{center}
\begin{tikzpicture} [scale=0.4]
 \begin{scope}
 \draw[fill=vert!60] (5,4) ellipse (5 and 4);
 \draw[vert,fill=gray!10] (5,4) ellipse (4.7 and 3.7);
 \draw[vert,fill=vert!60] (1.57,3.3) circle (0.3) (4.9,6.3) circle (0.3) (6,4.2) circle (0.3) (7.6,1.8) circle (0.3) (5.48,2.67) circle (0.3);
 \draw[very thick] (5,4) ellipse (5 and 4) (1,7.5) node {$\tilde\Sigma$};
 \draw (0.15,3) -- (9.85,5) (1,6.4) -- (2,0.8)
 (3,5) -- (4,2) (8,6) -- (9,4)
 (3,7.65) arc (-160:-17:2) (3.5,5.5) arc (180:10:1.5)
 (3,2.5) -- (5.5,3) -- (5,0)
 (3,1.5) circle (0.4) (7.5,2.5) circle (0.7);
 \draw[vert] (9,0.5) node {$\tilde C_1$};
 \end{scope}
 \begin{scope} [xshift=14cm]
 \draw[fill=vert!60] (5,4) ellipse (5 and 4);
 \draw[vert,fill=gray!10] (3,1.5) circle (0.6);
 \draw[vert,fill=gray!10,rounded corners=2pt] (3.2,2.35) -- (3.6,2.4) -- (3.5,2.8) -- (2.7,2.7) -- (2.8,2.3) -- (3.2,2.35)
 (4.57,2.6) -- (4.95,2.65) -- (4.85,3.05) -- (4.1,2.95) -- (4.2,2.55) -- (4.57,2.6)
 (3,4.1) -- (3.35,3.15) -- (3.85,3.2) -- (3.2,5.2) -- (2.7,5.1) -- (3,4.1) 
 (3.7,2.4) -- (3.9,1.8) -- (4.25,1.8) -- (3.8,3.1) -- (3.5,3) -- (3.7,2.4)
 (8,5.5) -- (8.2,5) -- (8.7,5.1) -- (8.2,6.2) -- (7.7,6.1) -- (8,5.5)
 (8.6,4.27) -- (8.8,3.8) -- (9.35,3.85) -- (9,4.6) -- (8.5,4.5) -- (8.6,4.27)
 (3.38,6) arc (160:105:1.7) -- (4.65,6.75) arc (105:185:1.3) -- (3.3,5.3) arc (185:160:1.7)
 (6.3,6.62) arc (40:0:1.7) -- (6.25,5.5) arc (0:80:1.3) -- (5.3,7.2) arc (80:40:1.7);
 \draw[very thick] (5,4) ellipse (5 and 4) (0.8,7.5) node {$\tilde\Sigma_1$};
 \draw (0.15,3) -- (9.85,5) (1,6.4) -- (2,0.8);
 \draw (3,5) -- (3.55,3.35) (3.7,2.9) -- (4,2) (8,6) -- (8.4,5.2) (8.8,4.4) -- (9,4);
 \draw (3,7.65) arc (-160:-17:2) (3.5,5.5) arc (180:115:1.5) (5.5,6.9) arc (70:10:1.5);
 \draw (2.9,2.5) -- (3.4,2.6) (4.25,2.75) -- (4.75,2.85) (5.5,3) -- (5,0);
 \draw  (3,1.1) arc (270:0:0.4) (7.5,2.5) circle (0.7);
 \draw[very thick,fill=white] (1.57,3.3) circle (0.3) (4.9,6.3) circle (0.3) (6,4.2) circle (0.3) (7.6,1.8) circle (0.3) (5.48,2.67) circle (0.3);
 \draw[vert] (9,0.5) node {$\tilde C_2$};
 \end{scope}
 \begin{scope} [xshift=28cm]
 \draw[very thick,fill=vert!60,yshift=-1cm] (3,1.5) circle (0.65);
 \draw[very thick,fill=vert!60,rounded corners=2pt,xshift=-0.8cm,yshift=-0.75cm,scale=1.1] (3.2,2.35) -- (3.6,2.4) -- (3.5,2.8) -- (2.7,2.7) -- (2.8,2.3) -- (3.2,2.35);
 \draw[very thick,fill=vert!60,rounded corners=2pt,xshift=0.03cm,yshift=-0.8cm,scale=1.1] (4.57,2.6) -- (4.95,2.65) -- (4.85,3.05) -- (4.1,2.95) -- (4.2,2.55) -- (4.57,2.6);
 \draw[very thick,fill=vert!60,rounded corners=2pt,yshift=-0.9cm,xshift=-0.8cm,scale=1.2] (3.7,2.4) -- (3.9,1.8) -- (4.25,1.8) -- (3.8,3.1) -- (3.5,3) -- (3.7,2.4);
 \draw[very thick,fill=vert!60,rounded corners=2pt] 
 (8.6,4.27) -- (8.8,3.8) -- (9.35,3.85) -- (9,4.6) -- (8.5,4.5) -- (8.6,4.27)
 (8,5.5) -- (8.2,5) -- (8.7,5.1) -- (8.2,6.2) -- (7.7,6.1) -- (8,5.5);
 \draw[very thick,fill=vert!60,rounded corners=3pt,xshift=-0.3cm] 
 (2.9,3.73) -- (3.15,2.95) -- (4.1,3) -- (3.4,5.4) -- (2.4,5.3) -- (2.9,3.73);
 \draw[vert,fill=gray!10,rounded corners=2pt,xshift=-0.3cm]
 (3,4.1) -- (3.35,3.15) -- (3.85,3.2) -- (3.2,5.2) -- (2.7,5.1) -- (3,4.1); 
 \draw[xshift=-0.3cm] (3,3.55) -- (3.9,3.7) (3,5) -- (3.55,3.35); 
 \begin{scope} [xshift=0.5cm]
 \draw[very thick,fill=vert!60,rounded corners=3pt]
 (3.15,6) arc (160:95:1.9) -- (4.75,6.5) arc (105:185:1.2) -- (3.1,5.1) arc (190:160:1.9)
 (6.54,6.74) arc (40:-5:1.9) -- (6.1,5.3) arc (-2:83:1.3) -- (5.1,7.4) arc (90:40:1.9); 
 \draw[vert,fill=gray!10,rounded corners=2pt]
 (3.38,6) arc (160:105:1.7) -- (4.65,6.75) arc (105:185:1.3) -- (3.3,5.3) arc (185:160:1.7)
 (6.3,6.62) arc (40:0:1.7) -- (6.25,5.5) arc (0:80:1.3) -- (5.3,7.2) arc (80:40:1.7);
 \draw (3.5,5.5) arc (180:115:1.5) (5.5,6.9) arc (70:10:1.5) (3.7,6.8) arc (-130:-105:2) (5.6,6.4) arc (-70:-43:2);
 \draw (1,7.5) node {$\tilde\Sigma_2$};
 \draw[vert] (5,0.5) node {$\tilde C_3$};
 \end{scope}
 \end{scope}
\end{tikzpicture}
\caption{Pushing a normal singular surface} \label{figpushcobborro}
\end{center}
\end{figure}

 \begin{samepage}
 Let $h:\tilde\Sigma\to[0,1]$ be a smooth function that sends:
 \begin{itemize}
  \item $\partial\tilde\Sigma$ onto $0$, $\partial\tilde\Sigma_1$ onto $\frac13$ and $\partial\tilde\Sigma_2$ onto $\frac23$,
  \item $\Int(\tilde C_1)$ to $(0,\frac13)$, $\Int(\tilde C_2)$ to $(\frac13,\frac23)$ and $\Int(\tilde C_3)$ to $(\frac23,1)$,
  \item $\tilde\Sigma_3$ onto $1$.
 \end{itemize}
 \end{samepage}
 Finally define an immersion $\iota':\tilde\Sigma\to S^3\times[0,1]$ by $\iota'(p)=\big(\iota(p),h(p)\big)$. Set $S=\iota'\left(\tilde\Sigma\setminus\Int(\tilde\Sigma_3)\right)$ and $\Delta=\iota'(\tilde\Sigma_3)$.
\end{proof}

The above results provide a characterization of the $T$--genus which generalizes a result of Kawauchi--Murakami--Sugishita in the case of knots \cite{KMS}.

\begin{corollary} \label{corcobborro}
 The $T$--genus of an algebraically split link $L$ is the smallest integer~$b$ such that there is a strict cobordism from $L$ to a split union of $b$ borromean links.
\end{corollary}
\begin{proof}
 If $\Sigma$ is a marked normal singular disks complex for $L$ with $b$ borromean triple points and no clasp, then Proposition~\ref{propcob} gives a strict cobordism from $L$ to a split union of $b$ borromean links. Reciprocally, given such a cobordism $S$, define $\Sigma$ as the disjoint union of disks bounded by the $b$ borromean links as in Figure~\ref{figDisksBorromean} and apply Theorem~\ref{thconcordance} to get a marked normal singular disks complex for $L$ with $b$ borromean triple points and no clasp.
\end{proof}

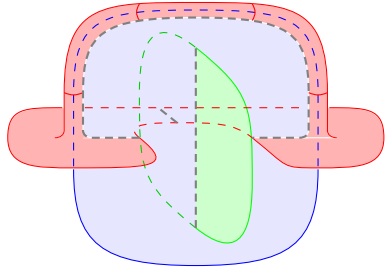
\begin{figure}[htb]
\begin{center}
\begin{tikzpicture} [xscale=0.25,yscale=0.2]
 \draw[white,fill=blue!10] (-6.5,-2.1) -- (-6.5,3) .. controls +(0,5) and +(-5,0) .. (0,8.5) .. controls +(5,0) and +(0,5) .. (6.5,3) -- (6.5,-2.1) .. controls +(0,-5) and +(5,0) .. (0,-8.5) .. controls +(-5,0) and +(0,-5) .. (-6.5,-2.1); 
 \draw[green,fill=green!20] (0,-6) .. controls +(2,-2) and +(0,-4) .. (3,-3) .. controls +(0,5) and +(3,-3) .. (0,6); 
 \draw[white,fill=red!30] (-7,2) .. controls +(-2,0) and +(0,2) .. (-10,0) .. controls +(0,-2) and +(-2,0) .. (-6.5,-2) .. controls +(4,0) and +(2,-2) .. (-3,0) -- (-6,0);
 \draw[white,fill=red!30] (7,2) .. controls +(2,0) and +(0,2) .. (10,0) .. controls +(0,-2) and +(2,0) .. (6.5,-2) .. controls +(-2,0) and +(2,-2) .. (3,0) -- (6,0);
 \draw[white,fill=red!30] (-7.5,0) .. controls +(0.3,0) and +(0,-0.3) .. (-7,0.5) -- (-7,3) .. controls +(0,6) and +(-5,0) .. (0,9) .. controls +(5,0) and +(0,6) .. (7,3) -- (7,0.5) .. controls +(0,-0.3) and +(-0.3,0) .. (7.5,0) -- (5.5,0) .. controls +(0.3,0) and +(0,-0.3) .. (6,0.5) -- (6,3) ..  controls +(0,4) and +(5,0) .. (0,8) .. controls +(-5,0) and +(0,4) .. (-6,3) -- (-6,0.5) .. controls  +(0,-0.3) and +(-0.3,0) .. (-5.5,0);
 \draw[vert,dashed] (0,-6) .. controls +(-3,3) and +(0,-5) .. (-3,3) .. controls +(0,4) and +(-2,2) .. (0,6); 
 \draw[red,dashed] (-5.9,2) -- (5.9,2);
 \draw[red,dashed] (-3,0) .. controls +(-1,1) and +(-1,0) .. (-1,1) .. controls +(2,0) and +(-1,1) .. (3,0);
 \draw[blue,dashed] (-6.5,-1.9) -- (-6.5,3) .. controls +(0,5) and +(-5,0) .. (0,8.5) .. controls +(5,0) and +(0,5) .. (6.5,3) -- (6.5,-1.9);
 \draw[red] (-7.1,2) .. controls +(-2,0) and +(0,2) .. (-10,0) .. controls +(0,-2) and +(-2,0) .. (-6.5,-2) .. controls +(4,0) and +(2,-2) .. (-3,0);
 \draw[red] (7.1,2) .. controls +(2,0) and +(0,2) .. (10,0) .. controls +(0,-2) and +(2,0) .. (6.5,-2) .. controls +(-2,0) and +(2,-2) .. (3,0);
 \draw[red] (-7.5,0) .. controls +(0.3,0) and +(0,-0.3) .. (-7,0.5) -- (-7,3) .. controls +(0,6) and +(-5,0) .. (0,9) .. controls +(5,0) and +(0,6) .. (7,3) -- (7,0.5) .. controls +(0,-0.3) and +(-0.3,0) .. (7.5,0);
 \draw[blue] (-6.5,-2.1) .. controls +(0,-5) and +(-5,0) .. (0,-8.5) ..controls +(5,0) and +(0,-5) .. (6.5,-2.1);
 \foreach \s in {-1,1} {
 \draw[red] (6*\s,3) .. controls +(0.4*\s,-0.2) and +(-0.4*\s,-0.2) .. (7*\s,3);
 \draw[red] (3*\s,7.8) .. controls +(0.2*\s,0.4) and +(0.2*\s,-0.4) .. (3*\s,8.9);}
 \draw[gray,densely dashed,thick] (-3,0) -- (-5.5,0) .. controls +(-0.3,0) and +(0,-0.3) .. (-6,0.5) -- (-6,3) .. controls +(0,4) and +(-5,0) .. (0,8) .. controls +(5,0) and +(0,4) .. (6,3) -- (6,0.5) .. controls +(0,-0.3) and +(0.3,0) .. (5.5,0) -- (3,0);
 \draw[gray,densely dashed,thick] (-1,1) -- (-2,2) (0,-6) -- (0,6);
\end{tikzpicture}
\caption{Ribbon complex for the borromean link} \label{figSeifertBorro}
\end{center}
\end{figure}

\begin{corollary} \label{corslicegenusTgenus}
 For any algebraically split link $L$, $g_s(L)\leq T_s(L)$.
\end{corollary}
\begin{proof}
 Figure~\ref{figSeifertBorro} shows that the borromean link bounds a slice complex of genus $1$. Take a strict cobordism from $L$ to a split union of $b=T_s(L)$ borromean links and complete it into a slice complex of genus $b$ for $L$ by gluing a genus--$1$ slice complex to each borromean link.
\end{proof}

\section{$T$--genera and $\Delta$--distance}

The $\Delta$--move on links represented in Figure~\ref{figdeltamove} was introduced by Murakami and Nakanishi \cite{MuNa}. Note that it preserves the linking numbers between the components of the link. The following result gives the converse \cite[Theorem~1.1]{MuNa}. 

\begin{theorem}[Murakami--Nakanishi] \label{thMuNa}
 Two links are related by a sequence of $\Delta$--moves if and only if they have the same number of components and their components have the same pairwise linking numbers.
\end{theorem}

A $\Delta$--move can be realized by {\em gluing a borromean link}, see Figure~\ref{figGlueBorro}. This will be useful for relating the $T$--genera and the invariants $\sd$ and $\rd$ defined in the introduction. The equality $T_s(K)=\sd(K)$ for a knot $K$ was given in \cite[Theorem 2]{KMS}.
\begin{figure}[htb] 
\begin{center}
\begin{tikzpicture}
\begin{scope} [yshift=0.8cm,scale=0.35]
 \foreach \t in {0,120,240} {
 \draw[rotate=\t] (-1.77,0.27) arc (135:350:2.5);
 \draw[rotate=\t] (2.46,-1.07) arc (10:105:2.5);}
 \draw (0,-6) node {Borromean link};
\end{scope}
\begin{scope} [xshift=7cm,scale=0.17]
 \foreach \t in {0,120,240} {
 \draw[rotate=\t] (-1.77,0.27) arc (135:225:2.5);
 \draw[rotate=\t] (1.77,-3.27) arc (315:350:2.5);
 \draw[rotate=\t] (2.46,-1.07) arc (10:105:2.5);
 \draw[rotate=\t,line width=10pt,white] (-14,-6) -- (-7,-6);
 \draw[rotate=\t] (-1.77,-3.27) .. controls +(3,-3) and +(3,0) .. (-7,-6) -- (-14,-6);
 \draw[rotate=\t] (1.77,-3.27) .. controls +(-3,-3) and +(-3,0) .. (7,-6) -- (14,-6);}
 \draw[line width=10pt,white] (-14,-6) -- (-8,-6);
 \draw (-14,-6) -- (-7,-6); 
 \draw (0,-13) node {Gluing a borromean link};
\end{scope}
\end{tikzpicture}
\end{center}
\caption{Borromean link and $\Delta$--move} \label{figGlueBorro}
\end{figure}
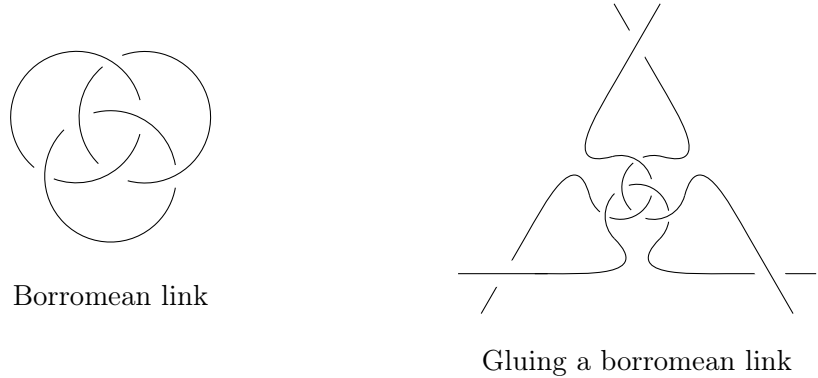

\begin{theorem} \label{thDeltaDistance}
 For any algebraically split link $L$, $T_s(L)=\sd(L)$ and $T_r(L)=\rd(L)$.
\end{theorem}
\begin{proof}
 The link $L$ can be obtained from a slice link, which bounds a marked normal singular disks complex $\Sigma$ with no borromean triple point and no clasp, by $\sd(L)$ $\Delta$--moves. Realize each of these $\Delta$--moves by gluing a borromean link. The borromean link bounds a complex of three disks that intersect along three ribbons, with a single borromean triple point, see Figure~\ref{figDisksBorromean}. Hence, when gluing a borromean link, we glue to $\Sigma$ such a complex of three disks with three bands. The bands may meet $\Sigma$ and create new ribbons, but we can assume that no other kind of singularity is added. Thus we obtain an immersed disks complex for $L$ with $\sd(L)$ borromean triple points and no clasp, proving $T_s(L)\leq \sd(L)$. Similarly, $L$ can be obtained from a ribbon link, which bounds a ribbon disks complex, by $\rd(L)$ $\Delta$--moves, so that $T_r(L)\leq\rd(L)$. It remains to prove the reverse inequalities. 
 
\begin{figure}[htb] 
\begin{center}
\begin{tikzpicture} 
\begin{scope}
 \draw (0.5,0) node {$\varnothing$} (4,0) circle (0.5);
 \node (A+) at (1,0.3) {};
 \node (A-) at (1,-0.3) {};
 \node (B+) at (3,0.3) {};
 \node (B-) at (3,-0.3) {};
 \draw[->] (A+) -- (B+) node[above,pos=0.5] {{\em birth}};
 \draw[->] (B-) -- (A-) node[below,pos=0.5] {{\em death}};
\end{scope}
\begin{scope} [xshift=9cm]
 \foreach \x in {0,5.4} {
 \draw[dashed] (\x-1.4,0.5) arc (90:270:0.5) (\x,0.5) arc (90:-90:0.5);}
 \draw (-1.4,0.5) arc (90:-90:0.5) (0,0.5) arc (90:270:0.5);
 \foreach \s in {-1,1} {
 \draw (4,0.5*\s) .. controls +(0.3,0) and +(-0.3,0) .. (4.7,0.1*\s) .. controls +(0.3,0) and +(-0.3,0) .. (5.4,0.5*\s);}
 \node (A+) at (1,0.3) {};
 \node (A-) at (1,-0.3) {};
 \node (B+) at (3,0.3) {};
 \node (B-) at (3,-0.3) {};
 \draw[->] (A+) -- (B+) node[above,pos=0.5] {{\em fusion}};
 \draw[->] (B-) -- (A-) node[below,pos=0.5] {{\em fission}};
\end{scope}
\end{tikzpicture}
\end{center}
\caption{Birth, death, fusion and fission moves} \label{figfusionfission}
\end{figure}
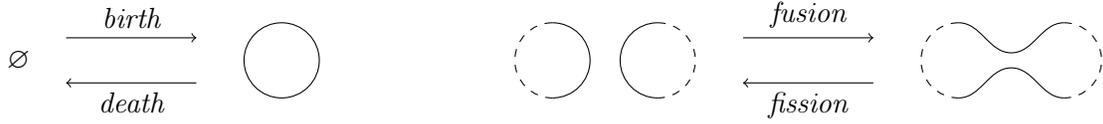 

 We shall prove that $T_s(L)$ can be reduced by gluing a borromean link. There is a strict cobordism $S$ from $L$ to the split union of $t=T_s(L)$ borromean links, which we denote $B^t$. This cobordism can be rearrange into the following pattern, where $\OO^k$ is a trivial $k$--component link, $J$ is a link and the different steps are described in Figure~\ref{figfusionfission}.
 \begin{center}
  \begin{tikzpicture} [xscale=1.5]
   \node (0) at (0,0) {$L$};
   \node (1) at (2,0) {$L\sqcup\OO^k$};
   \node (2) at (4,0) {$J$};
   \node (3) at (6,0) {$B^t\sqcup\OO^\ell$};
   \node (4) at (8,0) {$B^t$};
   \draw[->] (0) -- (1) node[below,pos=0.5] {{\em births}};
   \draw[->] (1) -- (2) node[below,pos=0.5] {{\em fusions}};
   \draw[->] (2) -- (3) node[below,pos=0.5] {{\em fissions}};
   \draw[->] (3) -- (4) node[below,pos=0.5] {{\em deaths}};
  \end{tikzpicture}
 \end{center}
 We can assume that each connected component of $S$ contains one component of $L$ and one component of $J$. One can get a trivial 3--component link as a band sum of two borromean links, see Figure~\ref{figBorroToTrivial}. Perform such a band sum gluing a borromean link $B$ to our link $B^t$. The bands gluing $B$ to $B^t$ can be glued before the fissions and deaths, hence be glued to $J$. Then these bands can be slid to be glued onto parts of the link $L$. Hence we can start by gluing $B$ to $L$ and then perform the remaining of the cobordism. This provides a strict cobordism from $L\sharp B$ to $B^{t-1}\sqcup\OO^3$. Filling in the components of $\OO^3$ with disks, we finally get a strict cobordism from the connected sum $L\sharp B$ to $B^{t-1}$. In other words, we get a cobordism from a link $L'$, obtained from $L$ by gluing a borromean link, to the disjoint union of $t-1$ borromean links; in particular $T_s(L')\leq T_s(L)-1$. It follows that $\sd(L)\leq T_s(L)$.
 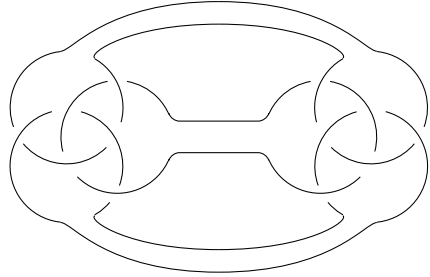
\begin{figure}[htb] 
 \begin{center}
 \begin{tikzpicture} [yshift=0.8cm,scale=0.3]
 \foreach \s in {1,-1} {
 \foreach \t in {90,210,330} {
 \draw[xshift=6*\s cm,xscale=-\s,rotate=\t] (-1.77,0.27) arc (135:350:2.5);
 \draw[xshift=6*\s cm,xscale=-\s,rotate=\t] (2.46,-1.07) arc (10:105:2.5);}
 \draw[white,line width=5pt] (2*\s,-1) -- (2*\s,1);
 \draw[white,line width=5pt] (5.5*\s,3.5) -- (7*\s,3.8);
 \draw[white,line width=5pt] (5.5*\s,-3.5) -- (7*\s,-3.8);}
 \foreach \s in {1,-1} {
 \draw[rounded corners=3pt] (-2.2,1*\s) -- (-2.1,0.7*\s) -- (2.1,0.7*\s) -- (2.2,1*\s);
 \foreach \t in {-1,1} {
 \draw (-7*\t,3.8*\s) .. controls +(0.5*\t,0) and +(-5*\t,0) .. (0,6*\s);
 \draw (-5.5*\t,3.5*\s) .. controls +(-0.3*\t,0.2*\s) and +(-4*\t,0) .. (0,5*\s);}}
 \end{tikzpicture}
 \end{center}
 \caption{A trivial $3$--component link as a band sum of two borromean links} \label{figBorroToTrivial}
 \end{figure}
 
 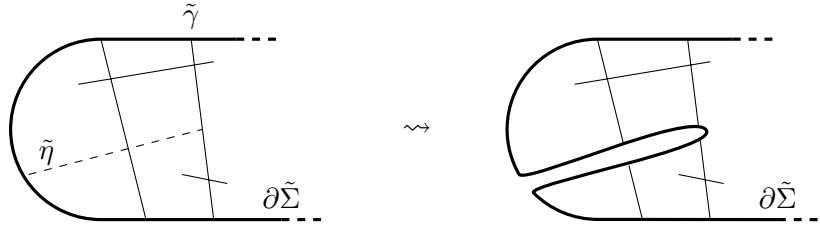
\begin{figure}[htb]
 \begin{center}
 \begin{tikzpicture} [scale=0.6]
  \newcommand{\tourr}{
  \draw[very thick,dashed] (4,-2) node[above] {$\partial\tilde\Sigma$} -- (5,-2) (3,2) -- (4,2);}
  \begin{scope}
  \tourr
  \draw[very thick] (3,2) -- (0,2) arc (90:270:2) -- (4,-2);
  \draw (0,2) -- (1,-2) (2,2) node[above] {$\tilde\gamma$} -- (2.5,-2) (-0.5,1) -- (2.5,1.5) (1.8,-1) -- (2.8,-1.2);
  \draw[dashed] (-1.6,-1) node[above right] {$\tilde\eta$} -- (2.25,0);
  \end{scope}
  \draw (7,0) node {$\rightsquigarrow$};
  \begin{scope} [xshift=11cm]
  \tourr
  \draw[very thick] (3,2) -- (0,2) arc (90:210:2) .. controls +(0.2,-0.2) and +(-0.4,0.4) .. (2.4,0) .. controls +(0.4,-0.4) and +(-0.2,0.2) .. (-1.41,-1.41) arc (225:270:2) -- (4,-2);
  \draw (0,2) -- (0.58,-0.32) (0.7,-0.8) -- (1,-2) (2,2) -- (2.24,0.08) (2.28,-0.24) -- (2.5,-2) (-0.5,1) -- (2.5,1.5) (1.8,-1) -- (2.8,-1.2);
  \end{scope}
 \end{tikzpicture}
 \caption{Sliding $\partial\Sigma$ along a path} \label{figSlideBoundary}
 \end{center}
 \end{figure}
 Now, take a $T$--ribbon disks complex $\Sigma$ for $L$ with $b=T_r(L)$ borromean triple points. Let $\tilde\Sigma$ be its preimage. Assume there is a ribbon $\gamma$ on $\Sigma$ which contains more than one triple point; denote $\tilde\gamma$ the corresponding $b$--line on $\tilde\Sigma$. Take a path $\tilde\eta$ on $\tilde\Sigma$ joining a point of $\partial\tilde\Sigma$ to a point of $\tilde\gamma$ between two preimages of triple points, such that $\tilde\eta$ does not meet any $i$--line --- which is possible since the $i$--lines are disjoint; denote $\eta\subset\Sigma$ the image of $\tilde\eta$. Slide the boundary of $\Sigma$ along $\eta$, see Figure~\ref{figSlideBoundary}. This results in an isotopy of $L$. Since $\tilde\gamma$ may intersect some $b$--lines, some ribbons may have been divided into more ribbons, but no other singularity appears. Moreover, $\gamma$ has been divided into ribbons with less triple points. Finally, we can assume that any ribbon on $\Sigma$ contains at most one triple point.
 
 Fix a triple point $p$ of $\Sigma$. For each preimage $\tilde p$ of $p$, take a neighborhood in $\tilde\Sigma$ of the union of the corresponding $i$--line with a part of the corresponding $b$--line joining $\tilde p$ to $\partial\tilde\Sigma$, see Figure~\ref{figRemoveTriplePoint}. These neighborhoods are three disks with only three ribbons meeting at a borromean triple point. Cutting these three disks amounts to cutting a borromean link. The same modification can be performed by a single $\Delta$--move. This implies $\rd(L)\leq T_r(L)$.
 \begin{figure}[htb]
 \begin{center}
 \begin{tikzpicture}
 \begin{scope} [xscale=0.5,yscale=0.3]
  \draw[gray!40,fill=gray!40] (1.5,0) .. controls +(0,2) and +(0,-1) .. (0.6,2.3) .. controls +(0,1) and +(-0.4,-0.2) .. (2,3.8) .. controls +(0.4,0.2) and +(0,1) .. (3.4,3.7) .. controls +(0,-1) and +(0,3) .. (2.5,0);
  \draw[very thick] (0,0) -- (4,0) (4,6) node[above] {$\partial\tilde\Sigma$} -- (0,6);
  \draw[very thick,dashed] (-1,0) -- (5,0) (5,6) -- (-1,6);
  \draw (2,6) -- (2,0) (1,2.5) -- (3,3.5);
 \end{scope}
 \draw (4,0.9) node {$\sim$};
 \begin{scope} [xshift=6cm,xscale=0.5,yscale=0.3]
  \draw[gray!40,fill=gray!40] (3,0) -- (1,0) .. controls +(1,0) and +(0,1) .. (1,-1.5) .. controls +(0,-1) and +(-1,0) .. (2,-3) .. controls +(1,0) and +(0,-1) .. (3,-1.5) .. controls +(0,1) and +(-1,0) .. (3,0);
  \draw[very thick] (0,0) -- (1,0) .. controls +(1,0) and +(0,1) .. (1,-1.5) .. controls +(0,-1) and +(-1,0) .. (2,-3) .. controls +(1,0) and +(0,-1) .. (3,-1.5) .. controls +(0,1) and +(-1,0) .. (3,0) -- (4,0) (4,6) node[above] {$\partial\tilde\Sigma$} -- (0,6);
  \draw[very thick,dashed] (-1,0) -- (1,0) (3,0) -- (5,0) (5,6) -- (-1,6);
  \draw (2,6) -- (2,-3) (1.5,-1.8) -- (2.5,-1.3);
 \end{scope}
 \end{tikzpicture}
 \caption{Isolate a borromean triple point} \label{figRemoveTriplePoint}
 \end{center}
 \end{figure}
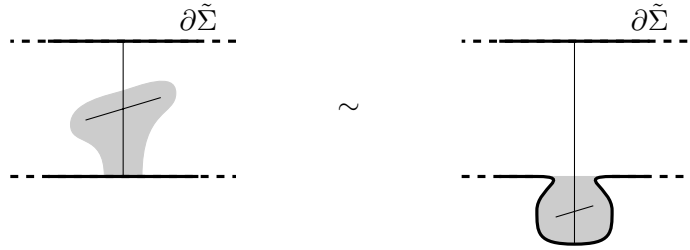
\end{proof}

We recover the result of Corollary~\ref{corslicegenusTgenus} and get the ribbon counterpart of it.
\begin{corollary} \label{corgenusTgenus}
 For any algebraically split link $L$, $g_s(L)\leq T_s(L)$ and $g_r(L)\leq T_r(L)$.
\end{corollary}
\begin{proof}
 By Theorem~\ref{thDeltaDistance}, $T_r(L)=\rd(L)$, so that $L$ can be obtained from a ribbon link, bounding a ribbon disks complex, by a sequence of $T_r(L)$ $\Delta$--moves. We have seen that a $\Delta$--move can be realized by gluing a borromean link. This can be achieved by gluing to the ribbon complex represented in Figure~\ref{figSeifertBorro}, which has genus~$1$. The same holds for the slice version.
\end{proof}

\section{Arf invariant and Milnor's triple linking number}

Recall that the Arf invariant is a $\Z/2\Z$--valued concordance invariant of knots which takes the value $0$ on the trivial knot and the value $1$ on the trefoil knot. The following result \cite[Theorem 2]{Robertello} allows to extend the Arf invariant to algebraically split links (Robertello gives this result more generally for the so-called proper links).

\begin{theorem}[Robertello]
 Let $L$ be an algebraically split link. If there are strict cobordisms from two knots $K$ and $K'$ to $L$, then the Arf invariants $\Arf(K)$ and $\Arf(K')$ are equal.
\end{theorem}

For an algebraically split link $L$, define the {\em Arf invariant} of $L$ as $\Arf(L)=\Arf(K)$ for any knot $K$ such that there is a strict cobordism from $K$ to $L$. The following result was established in \cite{KMS} in the case of knots.

\begin{proposition} \label{propArf}
 Let $L$ be an algebraically split link. Let $\Sigma$ be a marked normal singular disks complex for $L$ with no clasp and $b$ borromean triple points. Then $\Arf(L)=b\ mod\ 2$.
\end{proposition}
\begin{proof}
 Let $K$ be a knot obtained from $L$ by merging the component, so that there is a strict cobordism from $K$ to $L$. By Proposition~\ref{propcob}, there is a strict cobordism from $L$ to a split union $B^b$ of $b$ borromean links. Composing these cobordisms gives a strict cobordism from $K$ to $B^b$. Hence $\Arf(K)=\Arf(L)=\Arf(B^b)$. Now, the trefoil knot can be obtained from the borromean link by merging the three components, see Figure~\ref{figBorroToTrefoil}, hence $\Arf(B)=1$. This concludes since the Arf invariant is additive under connected sum and thus under split union.
\end{proof}

\begin{figure}[htb] 
\begin{center}
\begin{tikzpicture} [scale=0.35]
\begin{scope}
 \foreach \t in {90,210,330} {
 \draw[rotate=\t] (-1.77,0.27) arc (135:350:2.5);
 \draw[rotate=\t] (2.46,-1.07) arc (10:105:2.5);}
 \foreach \s in {1,-1} {
 \draw[white,line width=5pt] (0.5,3.5*\s) -- (-1,3.8*\s) (3,2*\s) -- (3.8,1*\s);
 \draw (-1,3.8*\s) .. controls +(0.5,0) and +(-1.3,0.7*\s) .. (3,5.2*\s) .. controls +(1.8,-\s) and +(-0.3,0.5*\s) .. (3.8,1*\s);
 \draw (0.5,3.47*\s) .. controls +(-0.3,0.2*\s) and +(-0.9,0.5*\s) .. (2.5,4.3*\s) .. controls +(1.3,-0.7*\s) and +(0.5,-0.4*\s) .. (3,2*\s);}
\end{scope}
 \draw (8,0) node {$\sim$};
\begin{scope} [xshift=15cm]
 \foreach \t in {90,210,330} {
 \draw[rotate=\t] (-1.8,0.3) .. controls +(-2.4,-2.6) and +(-1,-3.1) .. (1.1,0.1);}
\end{scope}
\end{tikzpicture}
\end{center}
\caption{A trefoil knot as a band sum of a borromean link} \label{figBorroToTrefoil}
\end{figure}
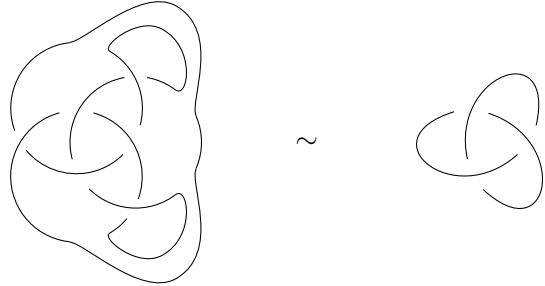

This result provides an interesting geometric realization of the Arf invariant of an algebraically split link $L$. The Arf invariant of $L$ is not only the parity of $T_s(L)$, but more generally the parity of the number of borromean triple points on any marked normal singular disks complex for $L$. This was first noticed by Kaplan for knots and $T$--ribbon disks \cite[Theorem 5.3]{Kaplan}.

We now give an expression of Milnor's triple linking number in terms of borromean triple points. The Milnor invariant $\mu(L)$ of an ordered $3$--component link $L=L_1\sqcup L_2\sqcup L_3$ is a $\Z$--valued concordance invariant introduced by Milnor; we refer the reader to \cite{Milnor} (see also \cite{Meilhan}) for its definition and main properties. Given a marked singular disks complex $D=D_1\cup D_2\cup D_3$ for $L$ with no clasp, define the {\em borromean triple point number} $n_{bt}(D)$ of $D$ as the number of borromean triple points of $D$ involving its three components, counted with signs. Note that changing the orientation of one component of $L$, as well as permuting two components, changes the sign of $n_{bt}(D)$; this also holds for $\mu(L)$.

\begin{proposition} \label{propMilnor}
 Let $L=L_1\sqcup L_2\sqcup L_3$ be an ordered algebraically split 3--component link. Let $D=D_1\cup D_2\cup D_3$ be a marked singular disks complex for $L$. Then $\mu(L)=n_{bt}(D)$.
\end{proposition}
\begin{proof}
 Proposition~\ref{propcob} provides a strict cobordism $S\subset S^3\times[0,1]$ from $L\subset S^3\times\{0\}$ to a split union $B^t$ of $t=T_s(L)$ borromean links bounding a $T$--ribbon complex $R^t\subset S^3\times\{1\}$, such that $S\cup R^t$ projects onto $D$. This cobordism can be worked out as in the proof of Theorem~\ref{thDeltaDistance} in order to get a strict cobordism, from a connected sum $L\sharp B$ to a sublink $B^{t-1}$ of $B^t$ made of $t-1$ borromean links, which projects to a marked singular disks complex for $L\sharp B$. Note that the sign of the remaining borromean triple points is unchanged. Note also that the connected components of $L$ involved in the connected sum correspond to the connected components of $D$ involved in the cancelled triple point.
 
 If the borromean link is glued along one or two components of the initial link, then the invariance of $\mu$ under link-homotopy --- a relation which allows isotopy and self-crossing change of each component --- shows that the value of $\mu(L)$ is unchanged; as is the value of $n_{bt}(D)$. If the three components of the borromean link are glued to the three components of the initial link, then it is a result of Krushkal \cite[Lemme 9]{Krushkal} that Milnor's triple linking number is additive under such a gluing. Hence $\pm1$ is added to $\mu(L)$, depending on the orientations; the same value is added to $n_{bt}(D)$.
 
 Repeat this operation until there is no more borromean triple point on $D$. At that stage, $n_{bt}(D)=0$ and $L$ has been turned to a slice link, so that $\mu(L)=0$ since $\mu$ is a concordance invariant. It follows that the initial values of $n_{bt}(D)$ and $\mu(L)$ were equal.
\end{proof}

\begin{corollary}
 Let $L=L_1\sqcup L_2\sqcup L_3$ be an ordered algebraically split 3--component link. Then $|\mu(L)|\leq T_s(L)$.
\end{corollary}

\begin{corollary} \label{corMilnor}
 Let $L=L_1\sqcup L_2\sqcup L_3$ be an ordered algebraically split 3--component link. Let $D=D_1\cup D_2\cup D_3$ be a $T$--ribbon disks complex for $L$. Then $\mu(L)$ equals the algebraic intersection number $\langle D_1,D_2,D_3\rangle$.
\end{corollary}

\section{Examples} \label{secex}

Let $B_n$ be the link obtained by 0--cabling $(n-1)$ times a component of the borromean link~$B$, see Figure~\ref{figlinkB4}. 
 
\begin{figure}[htb]
\begin{center}
\begin{tikzpicture}
 \foreach \x in {1,2,3,4} {\draw (\x,0) ellipse (0.2 and 0.5);}
 \draw[rounded corners=6pt] (0,0) -- (0,-1) -- (5,-1) -- (5,1) -- (0,1) -- (0,0);
 \draww{ (-0.1,0.3) arc (90:270:0.3) -- (0.3,-0.3) (5.1,0.3) arc (90:-90:0.3) -- (4.3,-0.3);}
 \foreach \x in {0,1,...,3} {\foreach \y in {-0.3,0.3} {
 \draww{(\x +0.3,\y) -- (\x +1,\y);}}}
 \draw (0.1,0.3) -- (0.3,0.3) (4.3,0.3) -- (4.9,0.3);
\end{tikzpicture}
\caption{The link $B_4$} \label{figlinkB4}
\end{center}
\end{figure}
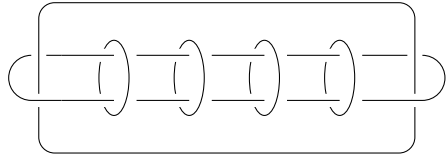

\begin{proposition}
 For all $n>0$, $T_r(B_n)=T_s(B_n)=n$ and $g_r(B_n)=g_s(B_n)=1$.
\end{proposition}
\begin{proof}
 First note that the borromean link $B$ satisfies $T_r(B)=T_s(B)=1$: Figure~\ref{figDisksBorromean} gives a $T$--ribbon disks complex for $B$ with one borromean triple point, and $B$ is not slice since $|\mu(B)|=1$.
 In Figure~\ref{figSeifertBorro}, taking $n$ parallel copies of the green disk provides a genus--1 ribbon complex for~$B_n$, so that $g_r(B_n)\leq1$. Moreover, $B_n$ is not slice since it has a non-slice sublink, namely $B$, thus $g_s(B_n)\geq1$.
 In Figure~\ref{figlinkB4}, the obvious disks bounded by the components of $B_n$ define a $T$--ribbon disks complex with exactly $n$ borromean triple points, giving $T_r(B_n)\leq n$. Now, take a marked normal singular disks complex $\mathcal D=D\cup D'\cup\left(\cup_{1\leq i\leq n}D_i\right)$ for $B_n$, where the disks $D_1,\dots,D_n$ are bounded by the parallel components of $B_n$. For any $1\leq i\leq n$, the number of borromean triple points on $\mathcal D$ defined by the intersection $D\cap D'\cap D_i$ is at least 1 since $B$ is not slice. Hence $T_s(B_n)\geq n$.
\end{proof}

Let $K_n$ be the {\em twist knot} defined by $n$ half-twists, see the left-hand side of Figure~\ref{figtwistknots}. It is easy to see that the genus of $K_n$ is~$1$. On the other hand, one can check that $K_0$ and $K_4$ are slice. A famous result of Casson and Gordon \cite{CG} states that these are the only slice twist knots.
\begin{figure}[htb]
\begin{center}
\begin{tikzpicture} [scale=0.4]
 \newcommand{\crossing}[2]{
 \begin{scope} [xshift=#1cm,yshift=#2cm]
 \draw (0,0) .. controls +(0,1) and +(0,-1) .. (1,2);
 \draww{(1,0) .. controls +(0,1) and +(0,-1) .. (0,2);}
 \end{scope}}
\begin{scope}
 \foreach \y in {0,2,...,6} {\crossing0\y}
 \draw (1,0) .. controls +(0,-2) and +(0,-7) .. (-2,4) .. controls +(0,4) and +(-1,0) .. (0,8.2) -- (0.8,8.2);
 \draww{(0,0) .. controls +(0,-2) and +(0,-7) .. (3,4) .. controls +(0,4) and +(1,0) .. (1.2,8.2);}
 \draw (1,8) -- (1,8.4) .. controls +(0,1) and +(0,1) .. (0,8.4);
 \draw (-1,3) node[rotate=90] {{\tiny $n$ half-twists}};
 \draw[gray] (-0.3,3) node {$\left\lbrace\resizebox{0cm}{1.8cm}{\phantom{rien}}\right.$};
\end{scope}
 \node (A) at (4,3.5) {}; \node (B) at (9,3.5) {};
 \draw[->] (A) -- (B) node[below,pos=0.5] {{\small $\Delta$--move}};
\begin{scope} [xshift=12.5cm]
 \foreach \y in {0,2,...,6} {\crossing0\y}
 \draw (1,0) .. controls +(0,-2) and +(0,-7) .. (-2,4) .. controls +(0,2) and +(-1,0) .. (-0.2,6);
 \draww{(0,0) .. controls +(0,-2) and +(0,-7) .. (3,4) .. controls +(0,2) and +(1,0) .. (1.2,6) -- (0.2,6);}
 \draw (1,8) .. controls +(0,1) and +(0,1) .. (0,8);
\end{scope}
 \draw (17,3.5) node {$=$};
\begin{scope} [xshift=20cm,yshift=1.5cm]
 \foreach \y in {0,2} {\crossing0\y}
 \draw (1,0) .. controls +(0,-1) and +(0,-4) .. (-1.5,2) .. controls +(0,2) and +(-1,0) .. (0,4.2) -- (0.8,4.2);
 \draww{(0,0) .. controls +(0,-1) and +(0,-4) .. (2.5,2) .. controls +(0,2) and +(1,0) .. (1.2,4.2);}
 \draw (1,4) -- (1,4.4) .. controls +(0,1) and +(0,1) .. (0,4.4);
\end{scope}
\end{tikzpicture}
\caption{A twist knot and a $\Delta$--move on it} \label{figtwistknots}
\end{center}
\end{figure}
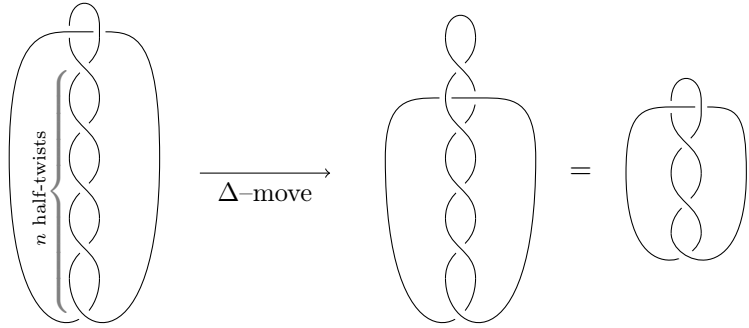

\begin{theorem}[Casson--Gordon]
 For $n\neq0,4$, $g_s(K_n)=1$.
\end{theorem}

\begin{lemma}
 For $n\neq0,4$, $c_4^b(K_n)=c_4(K_n)=1$.
\end{lemma}
\begin{proof}
 It is a simple and well-known fact that the unknotting number majors the 4--dimensional clasp number. Clearly, the unknotting number of non-trivial twist knots equals~$1$.
\end{proof}

\begin{proposition}
 For all $n\geq0$, $|T_r(K_{n+2})-T_r(K_n)|=|T_s(K_{n+2})-T_s(K_n)|=1$.
\end{proposition}
\begin{proof}
 For $n\geq0$, Figure~\ref{figtwistknots} shows that $K_{n+2}$ can be turned into $K_n$ by a single $\Delta$--move, which modifies $T_s$ and $T_r$ by at most one according to Theorem~\ref{thDeltaDistance}. Proposition~\ref{propArf} shows that they have to be modified (precisely it shows that their parity is changed).
\end{proof}

\begin{corollary}
 For all $n>1$, $T_r(K_{2n-1})\leq n$ and $T_r(K_{2n})\leq n-2$.
\end{corollary}
\begin{proof}
 The move on Figure~\ref{figtwistknots} changes $K_1$ into the trivial knot, so that $T_r(K_1)=1$. The knot $K_4$ is slice. 
\end{proof}

The following result shows the failure of Lemma~\ref{lemmagscb} for algebraically split links.
\begin{lemma}
 Let $L$ be the split union of a non-slice twist knot with its mirror image. Then $g_s(L)=2$ and $c_4^b(L)=1$. 
\end{lemma}
\begin{proof}
 The slice genus follows from that of non-slice twist knots. Since $L$ is not slice, $c_4^b(L)>0$. Now, each component of $L$ bounds a disk in $B^4$ with exactly one double point and the two disks can be chosen to be the mirror image of each other, in order to get two double points with opposite sign.
\end{proof}

This lemma implies that the difference $g_s-c_4^b$, and thus $T_s-c_4^b$, can be arbitrarily large for split links: take the split union of arbitrarily many copies of the link $L$ in the lemma.

\section{Colored links}

A {\em colored link} is a link $L$ in $S^3$ with a given partition into sublinks $L=L_1\sqcup\dots\sqcup L_\mu$. A {\em colored complex for $L$} is a union of compact surfaces $\Sigma=\cup_{1\leq i\leq \mu}\Sigma_i$ such that $\partial\Sigma_i=L_i$ for all $i$. We have as previously notions of normal singular, ribbon, $T$--ribbon, slice complex and a notion of marking. 

Given two links $K=\sqcup_{i=1}^kK_i$ and $J=\sqcup_{j=1}^\ell J_j$, where the $K_i$ and $J_j$ are knots, the {\em linking number of $K$ and $J$} is $\lk(K,J)=\sum_{i=1}^k\sum_{j=1}^\ell\lk(K_i,J_j)$. 
A colored link $L=\sqcup_{i=1}^\mu L_i$ is {\em algebraically $c$--split} if $\lk(L_i,L_j)=0$ for all $i\neq j$. We will generalize Kaplan's result, proving that any algebraically $c$--split colored link bounds a $T$--ribbon genus--$0$ colored complex. Although Kaplan's proof generalizes to this setting, we present here an alternative proof based on Theorem~\ref{thMuNa}. 
We start with preliminary results.

\begin{lemma} \label{lemmaLinkingsA}
 Fix a positive integer $n$ and integers $\ell_{ij}$ for $1\leq i<j\leq n$. There exists a link $K$ with $n$ connected components $K_i$ such that $\lk(K_i,K_j)=\ell_{ij}$ for all $i<j$ and $K$ bounds a genus--$0$ compact connected surface embedded in $S^3$.
\end{lemma}

\begin{proof}
 Take a trivial link $K$ with $n$ components $K_i$. It bounds an embedded genus--$0$ surface~$\Sigma$, compact and connected. For given $i<j$, the linking $\lk(K_i,K_j)$ can be modified as follows. Take a band $B=[0,1]\times[0,1]$ on $\Sigma$ such that $\{0\}\times[0,1]=B\cap K_i$, $\{1\}\times[0,1]=B\cap K_j$ and $(0,1)\times[0,1]\subset\Int(\Sigma)$. Twist the band $B$ around $\left\lbrace\frac12\right\rbrace\times[0,1]$, see Figure~\ref{figTwist}. This adds the number of twists to the linking number $\lk(K_i,K_j)$ without modifying the other linking numbers. 
\end{proof}
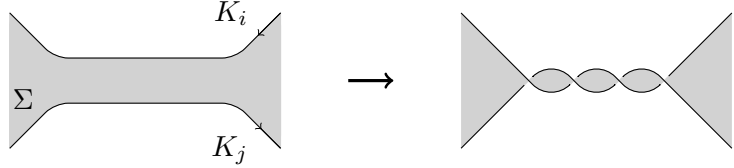
\begin{figure} [htb]
\begin{center}
\begin{tikzpicture} [xscale=0.6,yscale=0.3]
\begin{scope}
 \draw[color=gray!35,fill=gray!35,rounded corners=5pt] (0,-3) -- (1,-1) -- (5,-1) -- (6,-3) -- (6,3) -- (5,1) -- (1,1) -- (0,3) --(0,-3);
 \draw[color=gray!35,fill=gray!35] (0,-3) -- (1,-1) -- (5,-1) -- (6,-3) -- (6,3) -- (5,1) -- (1,1) -- (0,3) --(0,-3);
 \draw[rounded corners=5pt] (0,-3) -- (1,-1) -- (5,-1) -- (6,-3) (6,3) -- (5,1) -- (1,1) -- (0,3);
 \draw (0.3,-0.8) node {$\Sigma$};
 \draw[->] (6,3) -- (5.5,2) node[above left] {$K_i$};
 \draw[-<] (6,-3) -- (5.5,-2) node[below left] {$K_j$};
\end{scope}
\draw[->,very thick] (7.5,0) -- (8.5,0);
\begin{scope} [xshift=10cm]
 \draw[color=gray!35,fill=gray!35] (0,3) -- (1.5,0) -- (0,-3) (6,3) -- (4.5,0) -- (6,-3);
 \foreach \x in {0,1,2} {
 \draw[color=gray!35,fill=gray!35,rounded corners=10pt,xshift=\x cm] (1.5,0) -- (2,1) -- (2.5,0) (2.5,0) -- (2,-1) -- (1.5,0);}
 \draw[rounded corners=10pt] (0,-3) -- (1.4,-0.2) (1.6,0.2) -- (2,1) -- (3,-1) -- (3.4,-0.2) (3.6,0.2) -- (4,1) -- (6,-3) (6,3) -- (4.6,0.2) (4.4,-0.2) -- (4,-1) -- (3,1) -- (2.6,0.2) (2.4,-0.2) -- (2,-1) -- (0,3);
\end{scope}
\end{tikzpicture}
\caption{Twisting a band in $\Sigma$} \label{figTwist}
\end{center}
\end{figure}

\begin{lemma} \label{lemmaLinkingsB}
 Let $L=\sqcup_{1\leq i\leq\mu}L_i$ be an algebraically $c$--split colored link. There exists an algebraically $c$--split colored link $J=\sqcup_{1\leq i\leq\mu}J_i$ whose knot components have the same pairwise linking numbers as those of $L$ and which bounds a ribbon genus--$0$ colored complex in $S^3$.
\end{lemma}
\begin{proof}
 First, thanks to Lemma~\ref{lemmaLinkingsA}, define $J$ as the split union of links $J_i$ that realize the pairwise linking numbers of the $L_i$ and bound embedded genus--$0$ connected surfaces $\Sigma_i$ in $S^3$. Write $J_i$ as the disjoint union of knots $J_i=\sqcup_{1\leq \ell\leq k_i}J_{i\ell}$. Fix $i,j,\ell,m$ such that $1\leq i<j\leq\mu$, $1\leq\ell\leq k_i$ and $1\leq m<k_j$. Take a band in $\Sigma_j$ joining $J_{jm}$ to $J_{jk_j}$. Link this band around $J_{i\ell}$ in order to realize the desired linking $\lk(J_{i\ell},J_{jm})$, see Figure~\ref{figLink}. This adds ribbon intersections on the complex $\cup_{i=1}^\mu\Sigma_i$. Then, for $\ell<k_i$, realize the linking $\lk(J_{i\ell},J_{jk_j})$ using a band on $\Sigma_i$ joining $J_{i\ell}$ to $J_{ik_i}$. We have realized all the linking numbers $\lk(J_{i\ell},J_{jm})$ where $\ell<k_i$ or $m<k_j$. Since $J$ remains algebraically $c$--split, $\lk(J_{ik_i},J_{jk_j})$ is determined by these linking numbers, so that all linking numbers are finally realized.
\end{proof}

\begin{figure}[htb]
\begin{center}
\begin{tikzpicture} [xscale=0.6,yscale=0.3]
\begin{scope}
 \draw[color=gray!30,fill=gray!30] (0,0) -- (2,0) -- (2,10) -- (0,10);
 \draw (0,10) node[below right] {$\Sigma_i$};
 \draw[->] (2,10) -- (2,7) (2,0) -- (2,7) node[right] {$J_{i\ell}$};
 \draw[color=gray!50,fill=gray!50] (4,0) -- (5.5,0) -- (5.5,10) -- (4,10);
 \draw (4,10) node[below right] {$\Sigma_j$};
 \draw[->] (4,0) -- (4,3) (4,10) -- (4,3) node[left] {$J_{jm}$};
 \draw[->] (5.5,10) -- (5.5,5) (5.5,0) -- (5.5,5) node[right] {$J_{jk_j}$};
\end{scope}
\draw[->,very thick] (9,5) -- (10,5);
\begin{scope} [xshift=12.5cm]
 \draw[color=gray!50,fill=gray!50] (0.6,7.5) .. controls +(0,-1) and +(0,1) .. (4,5) -- (5.5,5) .. controls +(0,1) and +(0,-1) .. (1.6,7.5) 
 (0.6,2.5) .. controls +(0,-1) and +(0,1) .. (4,0) -- (5.5,0) .. controls +(0,1) and +(0,-1) .. (1.6,2.5);
 \draw (0.6,7.5) .. controls +(0,-1) and +(0,1) .. (4,5) (5.5,5) .. controls +(0,1) and +(0,-1) .. (1.6,7.5) 
 (0.6,2.5) .. controls +(0,-1) and +(0,1) .. (4,0) (5.5,0) .. controls +(0,1) and +(0,-1) .. (1.6,2.5);
 \draw[color=gray!30,fill=gray!30] (0,0) -- (2,0) -- (2,10) -- (0,10);
 \draw[dashed] (0.6,7.5) .. controls +(0,-1) and +(0,1) .. (4,5) (5.5,5) .. controls +(0,1) and +(0,-1) .. (1.6,7.5) 
 (0.6,2.5) .. controls +(0,-1) and +(0,1) .. (4,0) (5.5,0) .. controls +(0,1) and +(0,-1) .. (1.6,2.5);
 \draw[gray!50] (2,1.383) -- (2.1,1.336) (2,1.955) -- (2.1,1.89) (2,6.383) -- (2.1,6.336) (2,6.955) -- (2.1,6.89);
 \draw[gray!30,fill=gray!30] (1.9,0) -- (2,0) -- (2,10) -- (1.9,10);
 \draw (2,0) -- (2,2.8) (2,3.8) -- (2,7.8) (2,8.8) -- (2,10);
 \draw[color=gray!50,fill=gray!50] (4,10) .. controls +(0,-1) and +(0,1) .. (0.6,7.5) -- (1.6,7.5) .. controls +(0,1) and +(0,-1) .. (5.5,10) 
 (4,5) .. controls +(0,-1) and +(0,1) .. (0.6,2.5) -- (1.6,2.5) .. controls +(0,1) and +(0,-1) .. (5.5,5);
 \draw (4,10) .. controls +(0,-1) and +(0,1) .. (0.6,7.5) (1.6,7.5) .. controls +(0,1) and +(0,-1) .. (5.5,10) 
 (4,5) .. controls +(0,-1) and +(0,1) .. (0.6,2.5) (1.6,2.5) .. controls +(0,1) and +(0,-1) .. (5.5,5);
\end{scope}
\end{tikzpicture}
\caption{Linking a band in $\Sigma_j$ around a component of $J_i$} \label{figLink}
\end{center}
\end{figure}
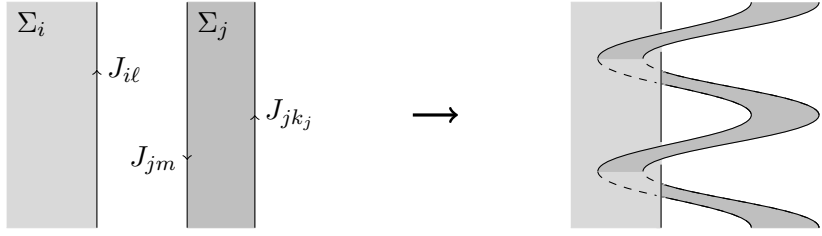

\begin{proposition} \label{propAlgTrivial}
 Any algebraically $c$--split colored link bounds a $T$--ribbon genus--$0$ colored \mbox{complex.}
\end{proposition}
\begin{proof}
 Let $L$ be an algebraically $c$--split colored link. Lemma~\ref{lemmaLinkingsB} provides another algebraically $c$--split colored link $J$ with the same pairwise linking numbers, that bounds a $T$--ribbon genus--$0$ colored complex $\Sigma$. Theorem~\ref{thMuNa} says that $L$ can be obtained from $J$ by a sequence of $\Delta$--moves. Realize these $\Delta$--moves by gluing borromean links to $L$ and associated $T$--ribbon disks complexes to $\Sigma$. This provides a ribbon disks complex for $L$.
\end{proof}

We consider a colored version of the invariants studied above, namely we define these invariants from colored complexes, and we add a superscript $c$ to distinguish them from the non-colored version. It follows from Proposition~\ref{propAlgTrivial} that the slice and ribbon genera and the $T$--genera of an algebraically $c$--split colored link are well-defined. Most of the results we have seen remain true in the colored setting, with the same proof. We collect them in the next statement.

A {colored cobordism} from a link $L$ to a link $L'$ is a disjoint union in $S^3\times[0,1]$ of genus--$0$ cobordisms from the colored sublinks of $L$ to the colored sublinks of $L'$. In this definition, some sublinks of the colored links may be empty. 

\begin{theorem}
 For any algebraically $c$--split colored link $L$,
\begin{itemize}
 \item the colored slice genus of $L$ equals the minimal genus of a marked normal singular colored complex for $L$ with no clasp intersection and no borromean triple point,
 \item the $T$--genus of $L$ is the smallest integer~$b$ such that there is a colored cobordism from $L$ to a split union of $b$ borromean links with any coloring,
 \item $T_s^c(L)=\sd^c(L)$ and $T_r^c(L)=\rd^c(L)$,
 \item $g_s^c(L)\leq T_s^c(L)$ and $g_r^c(L)\leq T_r^c(L)$,
 \item $c_4^{b,c}(L)\leq T_s^c(L)$.
\end{itemize}
\end{theorem}
\begin{proof}
 The first point is a corollary of Theorem~\ref{thprojection}.  
 The second point follows from Theorem~\ref{thconcordance} and Proposition~\ref{propcob}. 
 The proof of Theorem~\ref{thDeltaDistance} works in the colored setting and gives the third point; the fourth is a corollary of it. Finally the fifth point is a corollary of Proposition~\ref{propborrotoclasp}.
\end{proof}

Let $L$ and $L'$ be colored links. A {\em colored concordance} from $L$ to $L'$ is a disjoint union of concordances between the sublinks of $L$ and $L'$. Note that the relations in the next result are also satisfied by the slice genus.

\begin{theorem} 
 Let $L$ and $J$ be algebraically $c$--split colored links. Let $\widehat J$ be the colored link with the same underlying link as $J$ and a different color for each knot component. 
 \begin{itemize}
  \item If $L$ and $J$ are related by a colored cobordism, then $T_s(L)\leq T_s(\widehat J)$.
  \item If $L$ and $J$ are related by a colored concordance, then $T_s(L)=T_s(J)$.
 \end{itemize}
\end{theorem}
\begin{proof}
 First assume $L$ and $J$ are cobordant and define $\Sigma$ as the union of a cobordism from $L$ to $J$ with a marked normal singular disks complex $S$ for $\widehat J$ that realizes $T_s(\widehat J)$. Since $S$ is made of disks, after removing closed components if necessary, $\Sigma$ is a marked normal singular genus--$0$ complex for $L$ with at most $T_s(\widehat J)$ borromean triple points.
 
 Now assume $L$ and $J$ are concordant and do the same with a concordance from $L$ to $J$ and a marked normal singular genus--$0$ complex $S$ for $J$ that realizes $T_s(J)$. Once again, $\Sigma$ has genus~$0$, so that $T_s(L)\leq T_s(J)$. Similarly $T_s(J)\leq T_s(L)$.
\end{proof}

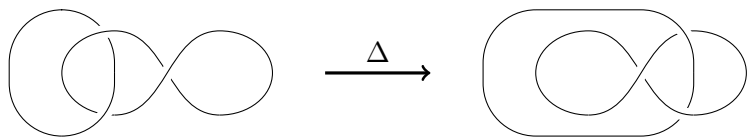
\begin{figure}[htb]
\begin{center}
\begin{tikzpicture} [scale=0.7]
\draw[rounded corners=20pt] (1,0) -- (1,1.2) -- (-1,1.2) -- (-1,-1.2) -- (0,-1.2);
\begin{scope} [yscale=0.8]
 \draww{(1,-1) arc (270:90:1) .. controls +(1,0) and +(-1,0) .. (3,-1);}
 \draww{ (3,-1) arc (-90:90:1) .. controls +(-1,0) and +(1,0) .. (1,-1);}
\end{scope}
\draww{[rounded corners=20pt] (0,-1.2) -- (1,-1.2) -- (1,0);}
\draw[very thick,->] (5,0) -- (6,0) node[above] {$\Delta$} -- (7,0);
\draw[rounded corners=20pt] (9,1.2) -- (8,1.2) -- (8,-1.2) -- (12,-1.2) -- (12,0);
\begin{scope} [xshift=9cm,yscale=0.8]
 \draww{ (1,-1) arc (270:90:1) .. controls +(1,0) and +(-1,0) .. (3,-1);}
 \draww{(3,-1) arc (-90:90:1) .. controls +(-1,0) and +(1,0) .. (1,-1);}
\end{scope}
\draww{[rounded corners=20pt] (12,0) -- (12,1.2) -- (9,1.2);}
\end{tikzpicture}
\caption{A $\Delta$--move on a Hopf link}  \label{figHopf}
\end{center}
\end{figure}

We end with a point which does not generalize to colored links. In the non-colored setting, the parity of the number of borromean triple points on a marked normal singular complex is fixed, it only depends on the link. A consequence of this fact is that a $\Delta$--move performed on an algebraically split link always modifies the link. Figure~\ref{figHopf} shows a $\Delta$--move on the Hopf link that leaves it unchanged. It follows that a colored Hopf link with a single color bounds $T$--ribbon genus--$0$ complexes with any number of borromean triple points.

\def\cprime{$'$}
\providecommand{\bysame}{\leavevmode ---\ }
\providecommand{\og}{``}
\providecommand{\fg}{''}
\providecommand{\smfandname}{\&}
\providecommand{\smfedsname}{\'eds.}
\providecommand{\smfedname}{\'ed.}
\providecommand{\smfmastersthesisname}{M\'emoire}
\providecommand{\smfphdthesisname}{Th\`ese}

\end{document}